\documentclass[11pt]{amsart}
\usepackage{amsfonts}
\usepackage{amsmath}
\usepackage{amssymb, amsthm}
\usepackage{bm}
\usepackage{comment}
\usepackage{enumitem}

\usepackage{mathtools}
\usepackage{xcolor}
\definecolor{Mylinkclr}{HTML}{0080FF}
\usepackage[pagebackref,colorlinks,linkcolor=Mylinkclr,citecolor=Mylinkclr]{hyperref}
   \allowdisplaybreaks
   \def\AA{\ensuremath{\mathcal A}}
\def\BB{\ensuremath{\mathcal B}}
\def\CC{\ensuremath{\mathcal C}}
\def\DD{\ensuremath{\mathcal D}}
\def\EE{\ensuremath{\mathcal E}}
\def\FF{\ensuremath{\mathcal F}}

\def\HH{\ensuremath{\mathcal H}}

\def\NN{\ensuremath{\mathcal N}}
\def\OO{\ensuremath{\mathcal O}}
\def\PP{\ensuremath{\mathcal P}}

\def\TT{\ensuremath{\mathcal T}}
\def\UU{\ensuremath{\mathcal U}}

\def\ch{\mathop{\mathrm{ch}}\nolimits}

\def\mc {\mathcal}
\def\mk {\mathfrak}
\def\mb {\mathbb}
\def\ol {\overline}

\newcommand{\xrightarrowdbl}[2][]{%
  \xrightarrow[#1]{#2}\mathrel{\mkern-14mu}\rightarrow
}

\usepackage{comment}
\usepackage{tikz-cd}
\usetikzlibrary{decorations.pathmorphing}

\makeatletter
\newtheorem*{rep@theorem}{\rep@title}
\newcommand{\newreptheorem}[2]{%
\newenvironment{rep#1}[1]{%
 \def\rep@title{#2 \ref{##1}}%
 \begin{rep@theorem}}%
 {\end{rep@theorem}}}
\makeatother
\tikzset{
  symbol/.style={
    draw=none,
    every to/.append style={
      edge node={node [sloped, allow upside down, auto=false]{$#1$}}}
  }
}
\newtheorem{lemma}{Lemma}[section]
\newreptheorem{lemma}{Lemma}
\newtheorem{theorem}[lemma]{Theorem}
\newreptheorem{theorem}{Theorem}
\newtheorem{corollary}[lemma]{Corollary}
\newreptheorem{corollary}{Corollary}
\newtheorem{prop}[lemma]{Proposition}
\newreptheorem{prop}{Proposition}

\theoremstyle{definition}
\newtheorem{defn}[lemma]{Definition}
\newtheorem{rem}[lemma]{Remark}

\theoremstyle{remark}
\newtheorem*{rem*}{Remark}
\newtheorem*{note*}{Note}

\newcommand\restr[2]{{
  \left.\kern-\nulldelimiterspace 
  #1 
  \vphantom{\big|} 
  \right|_{#2} 
  }}
\def\temp{&} \catcode`&=\active \let&=\temp

\newcommand{\Knum}{\mathcal N}
\DeclareMathOperator{\rank}{rk}

\begin{document}
\title{The stability manifold of $E{\times} E{\times} E$}

\author{Fabian Haiden}
\address{FH: Centre for Quantum Mathematics, Department of Mathematics and Computer Science, University of Southern Denmark, Campusvej 55, 5230 Odense, Denmark}
\email{fab@sdu.dk}

\author{Benjamin Sung}
\address{BS: Department of Mathematics, University of California, Santa Barbara, CA 93106, USA}
\email{bsung@ucsb.edu}

\date{}
\begin{abstract}
We determine a full component of the space of stability conditions on $D^b(E^3)$ where $E$ is an elliptic curve without complex multiplication. The component has complex dimension 14 and a very concrete description in terms of alternating trilinear forms.
This confirms a conjecture of Kontsevich, motivated by homological mirror symmetry, in the case of dimension $3$. 
\end{abstract}
\maketitle
\tableofcontents

\section{Introduction}

\subsection{Background}

Stability conditions on triangulated categories, introduced by Bridgeland in the seminal work~\cite{bridgeland07}, have attracted considerable interest due to their connections with a variety of subjects, including algebraic geometry, representation theory, and dynamics.
The main result of~\cite{bridgeland07} is the construction of a complex manifold structure on the space $\mathrm{Stab}(\mc T)$ of stability conditions on a triangulated category $\mc T$.
Determining the structure of this space, or just one of its components, is typically a difficult problem.

Of particular interest, because of their importance in Donaldson--Thomas theory and mathematical physics, is the case of 3-d Calabi--Yau triangulated categories $\mc T$.
Essentially the only cases where a full component of $\mathrm{Stab}(\mc T)$ has been determined are those related to spaces of quadratic differentials~\cite{BS15,haiden24,CHQ2}. 
In those cases, $\mc T$ is not of the form $D^b(X)$ for a Calabi--Yau 3-fold $X$.
On the other hand, in the influential work~\cite{MR3573975} on abelian 3-folds and crepant resolutions of their finite quotients, the authors restrict to those central charges depending only on the Chern numbers, which does not give the full space of stability conditions unless the Picard rank is 1.
In \cite{li19}, Li determines an open subset of $\mathrm{Stab}(D^b(X))$ for $X$ a smooth quintic 3-fold.

In this work we determine a connected component of the space $\mathrm{Stab}(D^b(X))$ when $X$ is an abelian 3-fold which is the cube of an elliptic curve without complex multiplication. (This space has $\dim_{\mb C}=14$.)
This is, to our knowledge, the first time a component of $\mathrm{Stab}(D^b(X))$ has been determined for a CY 3-fold $X$ of non-minimal Picard rank.
The statement is based on mirror symmetry and a proposal of Kontsevich on ``non-geometric'' stability conditions, i.e. those not coming from the complex structure/complexified K\"ahler moduli space.

\subsection{Main result}

In~\cite{haiden20}, based on a proposal of Kontsevich, a certain set $\mc U^+(V)$ of middle-degree complex-valued forms on a symplectic vector space $(V,\omega)$ was studied. We recall its definition.
First, let
\[
\mc U(V)\coloneqq\left\{\Omega\in{\bigwedge}^{n}V^*\otimes\mb C\ \middle| \ \Omega\wedge\omega=0,\Omega|_L\neq 0\text{ for any Lagrangian subspace }L\subset V\right\}
\]
where $2n=\dim V$.
Then $\mc U(V)$ has two components and we denote by $\mc U^+(V)$ the one which contains the forms of type $(n,0)$ with respect to some compatible complex structure. We write $\mc U^+(n)\coloneqq \mc U^+(\mb C^n)$.

A conjecture of Kontsevich~\cite{kontsevich_lille} proposes that $\mc U^+(n)$ is essentially the space of stability conditions on the category $D^b(E^n)$ where $E$ is an elliptic curve over $\mb C$ without complex multiplication.
We establish this conjecture in the case $n=3$ in the following form: 

\begin{theorem}\label{thm:main}
Let $E$ be an elliptic curve over $\mathbb{C}$ without complex multiplication and denote by $\Lambda\coloneqq\Knum(D^b(E^3))\cong\mb Z^{14}$ the numerical Grothendieck group.
Then there is a commutative diagram
\[
\begin{tikzcd}
\widetilde{\mc U^{+}(3)} \arrow[r, hook, "i"] \arrow[d, "\pi"] & \mathrm{Stab}\left(D^b(E^3)\right) \arrow[d]\\
\mc U^+(3) \arrow[r, hook, "j"] & \mathrm{Hom}(\Lambda, \mathbb{C})
\end{tikzcd}
\]
where $\pi$ is the universal covering and $i$ is a homeomorphism onto a connected component.
\end{theorem}

The analogous statement in the case $n=1$ is \cite[Theorem 9.1]{bridgeland07}, since $\mc U^+(1)=\mathrm{GL}^+(2,\mb R)$.
The case $n=2$ is also due to Bridgeland, a special case of~\cite[Theorem 15.2]{bridgeland08}.

\subsection{Mirror symmetry motivations}

The connection between $E^n$ and $\mc U^+(n)$ comes about via mirror symmetry.
The mirror of $E^n$ is the torus $T^{2n}=\mb R^{2n}/\mb Z^{2n}$ with standard symplectic form $\sum_{i=1}^ndx_i\wedge dy_i$, up to a scalar (the parameter $\tau$ of $E$).
Under the assumption that $E$ has no complex multiplication, there is an isomorphism $\Knum(D^b(E^n))\cong H_{n,\mathrm{pr}}(T^{2n};\mb Z)$, where the subscript, $\mathrm{pr}$, means we restrict to primitive classes --- those coming from Lagrangian cycles.
$H_{n,\mathrm{pr}}(T^{2n};\mb Z)$ is spanned by classes of linear Lagrangian sub-tori, of which there are many, thanks to our choice of symplectic form.
For $n=1,2,3,4,\ldots$ the rank of this lattice is equal to $2,5,14,42,\ldots$ --- the Catalan numbers.

It is a general expectation that for a compact Calabi--Yau $M$, there is a map from the complex structure moduli space to $\mathrm{Stab}(\mc F(M,\omega))$, where $\mc F(M,\omega)$ is the Fukaya category of $(M,\omega)$.
In particular, this assigns to a holomorphic volume form $\Omega$, i.e. a complex-valued middle-degree form which is of type $(n,0)$ with respect to some compatible complex structure, the central charge
\[
Z(L)\coloneqq\int_L\Omega.
\]
For $M=T^{2n}$, the space of choices of compatible complex structure and holomorphic volume form has dimension $n(n+1)/2+1$, which grows much more slowly than the exponentially growing Catalan numbers.
Thus, not all, in fact most stablility conditions on $D^b(E^n)=\mc F(T^{2n})$ can come from holomorphic volume forms.
For example, for $n=3$, the moduli space of $\Omega^{3,0}$'s has complex dimension $7$, only half of the expected dimension of the space of stability conditions.

The idea of Kontsevich is to replace the complex structure moduli space by $\mc U^+(n)$, whose dimension is equal to the rank of $\Knum(D^b(E^n))$, and thus has a chance to give all stability conditions on $D^b(E^n)$, or at least a component.
Going beyond the case of tori, one could consider closed complex-valued middle-degree forms which pointwise belong to $\mc U^+(n)$ as a soft replacement for holomorphic volume forms.

\subsection{Outline of argument}
This paper is organized into four parts, the sum of which culminates in a proof of Theorem~\ref{thm:main}.

Part~\ref{part:mirror} develops the mirror correspondence for abelian varieties following~\cite{msabelian}, and focuses on powers, $E^n$, of elliptic curves without complex multiplication. The combination of Proposition~\ref{prop:HodgePrimitive} and Lemma~\ref{lem:Umirror} establishes an $\mathrm{Sp}(2n,\mb Z)$-equivariant isomorphism $\Knum(D^b(E^n))\cong H_{n,\mathrm{pr}}(T^{2n};\mb Z)$, whose complexification identifies the space $\UU^+(n)$ as a subset of $\Knum(D^b(E^n))$. This clarifies the set of desired central charges, whose corresponding stability conditions will be constructed using algebro-geometric techniques in Part~\ref{part:bmodel}.

Part~\ref{part:amodel} exploits the symmetry group $\mathrm{Sp}(6) \times \mathrm{GL}^+(2,\mathbb{R})$ to describe a fundamental domain for the action on $\UU^+(3)$. Combined with the results of Part~\ref{part:mirror}, Proposition~\ref{prop:u3normalform} gives a simple description for the mirror of the fundamental domain as the following set of central charges:
\begin{equation}\label{eq:fund}\begin{split}
\Pi \coloneqq \{ \mathrm{exp}(i (D_1+ D_2 + D_3)) + &\alpha\,  \mathrm{exp}(i (-D_1 - D_2 + D_3)) + \beta\,  \mathrm{exp}(i (D_1-  D_2 - D_3)) \\ + &\gamma\,  \mathrm{exp}(i (-D_1+ D_2 - D_3)) 
\ \vert \ \alpha, \beta, \gamma \geq 0 , \ \alpha + \beta + \gamma < 1 \}.
\end{split}
\end{equation}
where $D_i$ denotes the divisor class of a point on the $i$th factor of $E^3$.

Part~\ref{part:bmodel} constructs, in earnest, a corresponding numerical stability condition satisfying the full support property for every central charge in $\Pi$. Proposition~\ref{prop:fullsupport} demonstrates that the construction of product stability conditions in~\cite{liu21}, combined with the action of auto-equivalences, gives a stability condition with central charge $\mathrm{exp}(i(D_1 + D_2 + D_3))$ satisfying the full support property. By relying on the deformation argument of~\cite{MR4023385}, Proposition~\ref{prop:quadratic} demonstrates that the associated quadratic forms are strong enough to imply the existence of a corresponding stability condition satisfying the full support property for every central charge in $\Pi$.

Part~\ref{part:proof} concludes the proof of Theorem~\ref{thm:main}. The combination of Parts~\ref{part:amodel} and~\ref{part:bmodel} imply the existence of a connected open neighborhood of stability conditions, whose set of central charges contain the fundamental domain $\Pi$. Proposition~\ref{prop:spendo} implies that a dense subgroup of $\mathrm{Sp}(2n,\mb R)$ lifts to an action on $\mathrm{Stab}(D^b(E^n))$, whose action on $\NN(D^b(E^n))$ agrees with the mirror of the natural $\mathrm{Sp}(2n,\mb R)$ action on $H_{n,\mathrm{pr}}(T^{2n};\mb Z)$. Combined with the result of~\cite{LPMSZ} asserting that numerical stability conditions on products of higher genus curves are uniquely characterized by their central charges up to the phase of the skyscraper sheaf, as summarized in Theorem~\ref{thm:injective}, the conclusion follows from a straightforward gluing argument.

\subsection{Relation to other works}
The results of~\cite[Theorem 1.4]{MR3573975} and~\cite[Theorem 1.3]{MR4469233} describe a $4$-complex dimensional submanifold of $\mathrm{Stab}(D^b(X))$, for $X$ any complex abelian threefold. In our case of $X = E^3$ and in terms of the fundamental domain of~(\ref{eq:fund}), described more succinctly as the cone
\[
\CC \coloneqq \{ (\alpha, \beta, \gamma) \in \mathbb{R}^3 \ \vert\  \alpha, \beta, \gamma \geq 0 , \ \alpha + \beta + \gamma < 1 \},
\]
this $4$-dimensional submanifold is mapped into the subset of $\CC$ satisfying the condition $\alpha = \beta = \gamma$. The result of~\cite[Lemma 3.19]{Oberdieck:2018uqa} then gives a quadratic form satisfying the full support property, yielding the existence of stability conditions in an open neighborhood of this line.

In contrast, our results in Part~\ref{part:bmodel} apply the construction of product stability conditions in~\cite[Corollary 1.2]{liu21}. We identify these stability conditions with the three outer rays of $\CC$ individually satisfying the conditions $\alpha = \beta = 0$, $\alpha = \gamma = 0$, and $\beta = \gamma = 0$. These subsets, in turn, are mirror to the product of primitive forms coinciding with the image of the map:
\begin{align*}
\UU^+(2) \times \UU^+(1) &\rightarrow \UU^+(3) \\
(\Omega_2, \Omega_1) &\mapsto p^*\Omega_2 \wedge q^* \Omega_1
\end{align*}
where $p$ and $q$ are the projection maps of the product $\mb C^3=\mb C^2\times\mb C$. To prove the full support property, we use the quadratic form of~\cite{liu21} together with a number of auto-equivalences to establish the existence of open neighborhoods of stability conditions containing these three rays. We prove that these quadratic forms are strong enough to establish the existence of stability conditions throughout the entire fundamental domain $\CC$.

We note that both our proof of the full support property in Proposition~\ref{prop:fullsupport} and our results in Part~\ref{part:proof} rely on Theorem~\ref{thm:injective} from~\cite{LPMSZ} which is unpublished at the time of writing, though the proof of it is complete.

\subsection{Further directions}

The problem of identifying a component of $\mathrm{Stab}(D^b(E^n))$ with $\widetilde{\UU^+(n)}$ remains open for $n\geq 4$. While some steps in our proof work for arbitrary $n$, some seem more specific to $n=3$ and new ideas are needed.

Staying in dimension three, we expect that one can generalize Theorem~\ref{thm:main} to those abelian varieties whose Hodge group contains an $\mathrm{SL}_2$ mirror to the Lefschetz $\mathrm{SL}_2$, since this ensures that Hodge cycles correspond to Lagrangian cycles in the mirror.
In the absence of such a condition, pursuing the doubling approach of~\cite{qin2022coisotropicbranestorihomological} could be fruitful. 

Our results provide a proof of concept for the following strategy towards establishing~\cite[Conjecture 1.7]{MR3573975} for higher dimensional abelian varieties whose Hodge group contains an $\mathrm{SL}_2$ mirror dual to the Lefschetz $\mathrm{SL}_2$.
\begin{enumerate}
    \item 
    Establish the existence of a stability condition on $E^n$ satisfying the support property with respect to a sufficiently high-dimensional subspace $\Lambda \subseteq \NN(E^n)\otimes \mathbb{R}$.
    \item 
    Use the results of~\cite{MR4292740,LPMSZ} to deform to a stability condition on the generic fiber of a deformation of $E^n$.
    \item 
    Use a dense subgroup of $\mathrm{Sp}(2n)$ and the results of Section~\ref{sec:extendspn} to generate the full stability manifold as a submanifold of $\UU^+(n)$.
\end{enumerate}

In addition, there are several aspects of our analysis that would be worth further investigating in detail. Our results in Section~\ref{sec:quadratic} establish the existence of stability conditions in the $3$-dimensional fundamental domain, $\Pi$, using only the existence of a stability condition corresponding to the vertex of the domain. This is consistent with the expectation that there are no wall crossings of the first kind in the sense of~\cite{kontsevich2008stabilitystructuresmotivicdonaldsonthomas} for abelian varieties. In particular, it strongly suggests that every stability condition on $D^b(E^3)$ has the same heart of bounded t-structure up to an $\mathrm{Sp}(6) \times \mathrm{GL}^+(2,\mathbb{R})$ action. Finally, it would be interesting to investigate the induced non-geometric stability conditions of $\UU^+(3)$ on Kummer-type constructions such as crepant resolutions of quotients of $E^{ 3}$ described in~\cite{Perry2024StabilityCO}.

\subsection{Acknowledgements}
We thank Dave Morrison, Xiaolei Zhao, and Emanuele Macr\`i for inspiring discussions. We thank Arend Bayer for comments on a preliminary version of this manuscript.
FH is supported by the VILLUM FONDEN, VILLUM Investigator grant 37814 and Sapere Aude grant 3120-00076B from the Independent Research Fund Denmark (DFF).
This paper is partly a result of the ERC-SyG project Recursive and Exact New Quantum Theory (ReNewQuantum) which received funding from the European Research Council (ERC) under the European Union's Horizon 2020 research and innovation programme under grant agreement No 810573.

\part{Mirror symmetry for powers of elliptic curves}\label{part:mirror}

\section{Mirrors of \texorpdfstring{$(p,q)$}{(p,q)}-classes}

We will be interested in the mirror, $(E^n)^\circ$, of $E^n$, where $E$ is an elliptic curve without complex multiplication.
For our purposes, we primarily view $E^n$ as a complex analytic torus and $(E^n)^\circ$ as a symplectic torus (with \textit{complexified} symplectic form $\omega_{\mb C}=B+i\omega$).

According to~\cite{msabelian}, the mirror relation can, in the case of tori, be defined in the following elementary terms.

\begin{defn}
A complex torus $V/\Gamma_V$ is \textit{mirror dual} to a torus $W/\Gamma_W$ with constant 2-form $B+i\omega$, where $\omega$ is non-degenerate, if the following holds.
There is an isomorphism of lattices 
\[
\alpha:\Gamma_V\oplus\Gamma_V^*\longrightarrow\Gamma_W\oplus\Gamma_W^*
\]
under which the complex structure on $V\oplus V^*$ corresponds to the complex structure on $W\oplus W^*$ given by
\[
J_{B+i\omega}\coloneqq\begin{pmatrix}
    \omega^{-1}B & -\omega^{-1} \\ \omega+B\omega^{-1}B & -B\omega^{-1}
\end{pmatrix}
\]
and where $\Gamma_V\oplus\Gamma_V^*$ is given the quadratic form coming from the natural pairing between $\Gamma_V$ and $\Gamma_V^*$, and likewise for $\Gamma_W\oplus\Gamma_W^*$.
\end{defn}

In the case where the complex torus is given in terms of normalized period coordinates, a mirror can be described more explicitly.

\begin{lemma}\label{lem:mirrorTorus}
Let $C=(c_{ij})$ be a complex $n$-by-$n$ matrix with $\mathrm{Im}(C)$ invertible. 
Then the complex torus $\mb C^n/\mb Z^n\oplus C\mb Z^n$ is mirror dual to the symplectic torus $\mb R^{2n}/\mb Z^{2n}$ with $\omega_{\mb C}=\sum_{i,j}c_{ij}dx_i\wedge dy_j$.
\end{lemma}

\begin{proof}
With respect to the canonical basis of the abelian group $\mb Z^n\oplus C\mb Z^n$ given by the standard $\mb C$-basis vectors of $\mb C^n$ together with the columns of $C=a+ib$, the complex structure is given by the matrix
\[
\begin{pmatrix} -ab^{-1} & -b-ab^{-1}a \\ b^{-1} & b^{-1}a \end{pmatrix}
\]
so the complex structure on $V\oplus V^*$, $V=\mb C^n$ is given by
\[
\begin{pmatrix}
    -ab^{-1} & -b-ab^{-1}a & 0 & 0 \\
    b^{-1} & b^{-1}a & 0 & 0 \\
    0 & 0 & b^{-T}a^T & -b^{-T} \\
    0 & 0 & b^T+a^Tb^{-T}a^T & -a^Tb^{-T} 
\end{pmatrix}.
\]
On the other hand, $J_{\omega_{\mb C}}$ is 
\[
\begin{pmatrix}
    b^{-T}a & 0 & 0 & b^{-T} \\
    0 & b^{-1}a & -b^{-1} & 0 \\
    0 & b + ab^{-1}a & -ab^{-1} \\
    -b^T-a^Tb^{-T}a^T & 0 & 0 & -a^Tb^{-T}
\end{pmatrix}
\]
The required isomorphism $\alpha$ is thus given by interchanging and switching the signs of the coordinates which make up the first and third blocks.
\end{proof}

As an example, if $E=\mb C/\mb Z\oplus\tau\mb Z$ is an elliptic curve, then a mirror $E^\circ$ is the torus $\mb R^2/\mb Z^2$ with $\omega_{\mb C}=\tau dx\wedge dy$.
Moreover, $(E^\circ)^n$ is mirror to $E^n$.

Let $A$ be a complex torus and $A^\circ$ its mirror in the above sense. Then there is a canonical, up to sign, isomorphism 
\[
\beta:H^\bullet(A;\mb Z)\to H^\bullet(A^\circ;\mb Z)
\]
of parity $\dim A$, see~\cite[Sections 9.3, 10.6]{msabelian}.
In the simplest case, $A=E=\mb C/\mb Z\oplus\tau\mb Z$, $\tau=a+bi\in\mb H$, the correspondence looks as follows:
\begin{center}
\renewcommand{\arraystretch}{1.4}\setlength{\tabcolsep}{8pt}
\begin{tabular}{c|cccc}
 $H^\bullet(E;\mb Z)$       & $1$  & $dx-\frac{a}{b}dy$ & $\frac{1}{b}dy$ & $\frac{1}{b}dx\wedge dy$\\
 \hline
 $H^\bullet(E^\circ;\mb Z)$ & $dx$ & $1$                & $dx\wedge dy$     & $dy$ 
\end{tabular}    
\end{center}
Taking the $n$-th tensor power composed with the K\"unneth isomorphism then gives an explicit description of the isomorphism
\begin{equation}\label{eq:totalHiso}
\beta:H^\bullet(E^n;\mb Z)\cong \left(H^\bullet(E;\mb Z)\right)^{\otimes n}\longrightarrow \left(H^\bullet(E^\circ;\mb Z)\right)^{\otimes n}\cong H^\bullet((E^n)^\circ;\mb Z).
\end{equation}
Assuming $E$ has no complex multiplication, this restricts to a correspondence between Hodge classes on $E^n$ and primitive middle degree classes on $(E^\circ)^n$.

\begin{prop}\label{prop:HodgePrimitive}
The isomorphism \eqref{eq:totalHiso} restricts to an isomorphism
\begin{equation}\label{eq:invHiso}
\bigoplus_{p=0}^nH^{2p}(E^n;\mb Q)\cap H^{p,p}(E^n)\longrightarrow H^n_{\mathrm{pr}}((E^\circ)^n;\mb Q)
\end{equation}
over $\mb Q$.
Moreover, the left-hand-side is isomorphic to $\Knum(D^b(E^n))\otimes\mb Q$ under the Chern character.
\end{prop}

\begin{proof}
Consider the Lie algebra $\mk{sl}_2=\mk{sl}(2,\mb C)$ with the usual basis $\{e,f,h\}$ such that $[e,f]=h$, $[h,e]=2e$, $[h,f]=-2f$.
This acts on $H^\bullet(E^n;\mb C)$ by algebra derivations defined on generators by $e(dx_i)=dy_i$, $e(dy_i)=0$, $f(dy_i)=dx_i$, $f(dx_i)=0$, $h(dx_i)=-dx_i$, $h(dy_i)=dy_i$.
Since the corresponding Lie group $\mathrm{SL}_2$ is precisely the Hodge group of $E^n$ (see~\cite[Appendix B.3]{Lewis1999ASO}, after \cite{Murty}), and the classes fixed under the Hodge group are the Hodge classes~\cite[Theorem 17.3.3]{Birkenhake1992ComplexAV}, the invariant rational classes are the left-hand-side of~\eqref{eq:invHiso}.

Under the isomorphism $\beta$, $\mk{sl}_2$ also acts on $H^\bullet((E^n)^\circ;\mb C)$ and this action is such that $e$ acts by the Lefschetz operator $\alpha\mapsto\mathrm{Im}(\omega^\circ)\wedge\alpha$ and $h$ acts on forms of degree $k$ by multiplication with $k-n$.
In particular, the subspace fixed under the $\mk{sl}_2$-action is by definition $H^n_{\mathrm{pr}}((E^\circ)^n;\mb C)$, whose intersection with rational cohomology is precisely the right-hand-side of~\eqref{eq:invHiso}.

The second part of the proposition, i.e. the Hodge conjecture for $E^n$, is also found in~\cite[Appendix B.3]{Lewis1999ASO}.
\end{proof}

Note that 
\begin{align*}
    \rank H^n_{\mathrm{pr}}((E^\circ)^n;\mb Z)&=\rank H^n((E^\circ)^n;\mb Z)-\rank H^{n+2}((E^\circ)^n;\mb Z) \\
    &=\binom{2n}{n}-\binom{2n}{n+2}=\frac{1}{n+2}\binom{2n+2}{n+1}
\end{align*}
which is the $(n+1)$-st Catalan number.

\subsection{The correspondence for $E\times E$}
By the assumption that $E$ has no CM, the lattice $\Knum(E\times E)$ has rank five.
The isomorphism from Proposition~\ref{prop:HodgePrimitive} is explicitly given as follows.
\begin{center}
\renewcommand{\arraystretch}{1.4}\setlength{\tabcolsep}{8pt}
\begin{tabular}{ccc}
$H_{\mathrm{even}}(E\times E;\mb Z)$ & $\bigoplus_{p}H^{p,p}(E\times E)$ & $H^2_{\mathrm{pr}}(E^\circ\times E^\circ)$ \\ 
\hline
 $\{(0,0)\}$ & $\frac{1}{b^2}dx_1\wedge dy_1\wedge dx_2\wedge dy_2$ & $dy_1\wedge dy_2$ \\
 $D_1\coloneqq 0\times E$ & $\frac{1}{b}dx_1\wedge dy_1$ & $dy_1\wedge dx_2$ \\
 $D_2\coloneqq E\times 0$ & $\frac{1}{b}dx_2\wedge dy_2$ & $dx_1\wedge dy_2$ \\
 $\Delta$ & \begin{tabular}{c}
$\frac{1}{b}(dx_1\wedge dy_1+dx_2\wedge dy_2$ \\ $+dx_1\wedge dy_2-dy_1\wedge dx_2)$  \end{tabular} & \begin{tabular}{c} $dy_1\wedge dx_2+dx_1\wedge dy_2$ \\ $+dx_2\wedge dy_2-dx_1\wedge dy_1$ \end{tabular} \\
 $E\times E$ & $1$ & $dx_1\wedge dx_2$
\end{tabular}
\end{center}

Here, $\Delta\subset E\times E$ is the diagonal. We can re-write this in another basis to restore compatibility with the decomposition of $H^2_{\mathrm{pr}}(E^\circ\times E^\circ;\mb C)$ into types $(2,0)$, $(1,1)$, and $(0,2)$.
\begin{center}
\renewcommand{\arraystretch}{1.4}\setlength{\tabcolsep}{8pt}
\begin{tabular}{cc}
$H_{\mathrm{even}}(E\times E;\mb C)$ & $H^2_{\mathrm{pr}}(E^\circ\times E^\circ;\mb C)$ \\ 
\hline
$\exp(i(D_1+D_2))$ & $dz_1\wedge dz_2$ \\
$\exp(i(D_1-D_2))$ & $dz_1\wedge d\bar{z}_2$ \\
$\exp(i(D_2-D_1))$ & $d\bar{z}_1\wedge dz_2$ \\
$2i(\Delta-D_1-D_2)$ & $dz_1\wedge d\bar{z}_1+d\bar{z}_2\wedge dz_2$ \\
$\exp(-i(D_1+D_2))$ & $d\bar{z}_1\wedge d\bar{z}_2$ 
\end{tabular}
\end{center}

\subsection{The correspondence for $E^3$}\label{subsec:E3corr}

The lattice $\Knum(E^3)$ has rank 14 and is generated by products of classes coming from $E$ and $E\times E$.
The isomorphism from Proposition~\ref{prop:HodgePrimitive} is explicitly given as follows. 
The table below only includes classes of types $(3,0)$ and $(1,2)$, the others are obtained by complex conjugation.
\begin{center}
\renewcommand{\arraystretch}{1.4}\setlength{\tabcolsep}{8pt}
\begin{tabular}{cc}
$H_{\mathrm{even}}(E^3;\mb C)$ & $H^3_{\mathrm{pr}}((E^\circ)^3;\mb C)$ \\ 
\hline
$\exp(i(D_1+D_2+D_3))$ & $dz_1\wedge dz_2\wedge dz_3$ \\
$\exp(i(D_1-D_2-D_3))$ & $dz_1\wedge d\bar{z}_2\wedge d\bar{z}_3$ \\
$\exp(i(-D_1+D_2-D_3))$ & $d\bar{z}_1\wedge dz_2\wedge d\bar{z}_3$ \\
$\exp(i(-D_1-D_2+D_3))$ & $d\bar{z}_1\wedge d\bar{z}_2\wedge dz_3$ \\
$2i F_{12} + 2D_3 \cdot F_{12}$ & $(dz_1\wedge d\bar{z}_1+d\bar{z}_2\wedge dz_2)\wedge d\bar{z}_3$ \\
$2i F_{13} + 2D_2 \cdot F_{13}$ & $(dz_1\wedge d\bar{z}_1+d\bar{z}_3\wedge dz_3)\wedge d\bar{z}_2$ \\
$2i F_{23} + 2D_1 \cdot F_{23}$ & $(dz_2\wedge d\bar{z}_2+d\bar{z}_3\wedge dz_3)\wedge d\bar{z}_1$
\end{tabular}
\end{center}
Here, $D_i\coloneqq\{z_i=0\}\subset E^3$, $\Delta_{ij}\coloneqq \{z_i = z_j\} \subset E^3$, and $F_{ij}\coloneqq \Delta_{ij} - D_i - D_j$.

\section{Mirror of \texorpdfstring{$\mathrm{Sp}(2n)$}{Sp(2n)}}\label{sec:mirrorspn}

Let $A$ be an abelian variety over an algebraically closed field $\mathbf k$, $\mathrm{char}(\mathbf k)=0$, and $\hat{A}=\mathrm{Pic}^0(A)$ its dual abelian variety.
According to a theorem of Orlov~\cite{orlov02}, the group $\mathrm{Aut}(D^b(A))$ fits into a short exact sequence
\begin{equation}\label{eq:AutAseq}
0\longrightarrow \mb Z\times A\times \hat{A}\longrightarrow\mathrm{Aut}(D^b(A))\longrightarrow U(A)\to 1
\end{equation}
where $\mb Z$ acts by shift, $A$ acts by translation, $\hat A$ acts by tensoring with the corresponding line bundles, and $U(A)$ is the group of isometric automorphisms of $A\times \hat{A}$.
More concretely,
\[
U(A)\coloneqq \left\{T=\begin{pmatrix} a & b \\ c & d \end{pmatrix}\middle|\: T^{-1}=\begin{pmatrix} \hat d & -\hat b \\ -\hat c & \hat a \end{pmatrix}\right\}
\]
where $a\in\mathrm{End}(A)$, $b\in\mathrm{Hom}(\hat A,A)$, $c\in\mathrm{Hom}(A,\hat A)$, $d\in\mathrm{End}(\hat A)$.

The homomorphism $\mathrm{Aut}(D^b(A))\to U(A)$ maps $F\mapsto f$ iff
\[
\Phi_{f(x,y)}\circ F=F\circ\Phi_{(x,y)},\qquad (x,y)\in A\times\hat A
\]
where $\Phi_{(x,y)}(\mc E)=t_x(\mc E)\otimes P_y$, $t_x$ being translation by $x$ and $P_y$ being the line bundle corresponding to $y$.
In particular, for $g\in\mathrm{Aut}(A)$ and $L\in\mathrm{Pic}(A)$ 
\[
g^*\mapsto \begin{pmatrix} g^{-1} & 0 \\ 0 & \hat{g} \end{pmatrix},\qquad -\otimes L\mapsto \begin{pmatrix} 1 & 0 \\ -\phi(L) & 1 \end{pmatrix}
\]
with $\phi(L):A\to\hat A$, $x\mapsto t^*_xL\otimes L^{-1}$ depending only on the class of $L$ in $NS(A)$.

As shown in~\cite[Proposition 4.3.4]{msabelian}, $\mathrm{Aut}(D^b(A))$ also naturally acts (by even automorphisms) on $H^\bullet(A;\mb Z)$ in a way compatible with $\mathrm{ch}:K_0(D^b(A))\to H^{\mathrm{ev}}(A;\mb Z)$. 
Define
\[
\mathrm{Spin}(A)\coloneqq \mathrm{im}\left(\mathrm{Aut}D^b(A)\to \mathrm{GL}(H^\bullet(A;\mb Z))\right)
\]
then there is a short exact sequence
\[
0 \longrightarrow \mb Z/2\longrightarrow\mathrm{Spin}(A)\longrightarrow U(A) \longrightarrow 1,
\]
see \cite[Corollary 4.3.9]{msabelian}.
In particular, $U(A)$ acts on $H^\bullet(A;\mb Z)$ up to a sign ambiguity, which will not present an issue in our considerations.

We specialize to the case $A=E^n$ where $E=\mb C/\mb Z\oplus\tau\mb Z$ is an elliptic curve without complex multiplication, and thus $\mathrm{End}(E)=\mb Z$.
There is a standard principal polarization $\varphi:A\to\hat A$ represented by the block matrix $\begin{pmatrix} 0 & 1 \\ -1 & 0 \end{pmatrix}$.

\begin{lemma}\label{lem:Umirror}
Under the mirror symmetry isomorphism $\alpha$ of Lemma~\ref{lem:mirrorTorus}, the action of $U(A)$ on $H_1(A\times\hat{A};\mb Z)=\mb Z^{4n}$ is intertwined with the action of $\mathrm{Sp}(H_1(A^\circ;\mb Z))=\mathrm{Sp}(2n,\mb Z)$ on $H_1(A^\circ\times\hat{A}^\circ;\mb Z)$ under an appropriate identification of these two groups.
\end{lemma}

\begin{proof}
Under the natural faithful representation
\[
U(A)\longrightarrow GL\left(H_1(A\times\hat{A};\mb Z)\right)
\]
elements of $U(A)$ correspond to block matrices of the form
\[
\begin{pmatrix}
    d & 0 & 0 & c \\
    0 & d & -c & 0 \\
    0 & -b & a & 0 \\
    b & 0 & 0 & a
\end{pmatrix}
\]
such that
\[
a^Tc=c^Ta,\qquad b^Td=d^Tb,\qquad a^Td-c^Tb=1.
\]
After conjugation by the isomorphism $\alpha$ of Lemma~\ref{lem:mirrorTorus}, given by exchanging the first and third blocks and switching their sign, these go to
\[
\begin{pmatrix}
    a & b & 0 & 0 \\
    c & d & 0 & 0 \\
    0 & 0 & d & -c \\
    0 & 0 & -b & a
\end{pmatrix}
\]
which is precisely the image of $\begin{pmatrix} a & b \\ c & d \end{pmatrix}\in\mathrm{Sp}(2n,\mb Z)$ under the natural representation on $H_1(A^\circ\times\hat{A}^\circ;\mb Z)$.
\end{proof}

Since $A$ is principally polarized by $\varphi:A\to\hat A$, we have an autoequivalence $\Phi\in\mathrm{Aut}(D^b(A))$ given by composition of the Fourier--Mukai functor $D^b(A)\to D^b(\hat{A})$ with $\varphi^*$ (see also Section~\ref{subsec:poincarebundle}).
Then 
\[
\Phi\mapsto\begin{pmatrix} 0 & -\hat\varphi^{-1} \\ \varphi & 0 \end{pmatrix}
\]
under the map $\mathrm{Aut}(D^b(A))\to U(A)$.

Since we have already determined the image of various autoequivalences in $U(A)$, we can combine this with Lemma~\ref{lem:Umirror} to obtain the following:

\begin{prop}
The autoequivalences $\otimes L$, $f^*$, $f_*$, and $\Phi$ of $D^b(A)$, where $L\in\mathrm{Pic}(A)$ and $f\in\mathrm{Aut}(A)$, have the following image in $\mathrm{Sp}(2n,\mb Z)$:
\[
\otimes L\mapsto \begin{pmatrix}
    1 & c_1(L) \\ 
    0 & 1
\end{pmatrix}, 
\quad 
f^* \mapsto \begin{pmatrix}
    f^T & 0 \\
    0 & f^{-1}
\end{pmatrix},
\quad 
f_* \mapsto \begin{pmatrix}
    f^{-T} & 0 \\
    0 & f
\end{pmatrix},
\quad
\Phi\mapsto \begin{pmatrix}
0 & -1\\
1 & 0
\end{pmatrix}.
\]
\end{prop}

Following Polishchuk, we will also consider the bigger group
\[
U(A,\mb Q)\coloneqq \left\{T=\begin{pmatrix} a & b \\ c & d \end{pmatrix}\middle|\: T^{-1}=\begin{pmatrix} \hat d & -\hat b \\ -\hat c & \hat a \end{pmatrix}\right\}
\]
where $a\in\mathrm{End}(A)\otimes\mb Q$, $b\in\mathrm{Hom}(\hat A,A)\otimes\mb Q$, $c\in\mathrm{Hom}(A,\hat A)\otimes\mb Q$, $d\in\mathrm{End}(\hat A)\otimes\mb Q$.
Thus, $U(A,\mb Q)$ is defined like $U(A)$ but with the category of abelian varities replaced by its localization at isogenies.
In the case $A=E^n$ where $E$ is without CM we have $U(A,\mb Q)=\mathrm{Sp}(2n,\mb Q)$.

Elements of $U(A,\mb Q)$ correspond to certain endofunctors of $D^b(A)$, see \cite{polishchuk2011lagrangianinvariantsheavesfunctorsabelian} and \cite[Section 2.1]{Polishchuk_2014}.
We will need the following special cases.

\begin{prop}\label{prop:endomirror}
The endofunctors $\otimes V$, $f^*$, $f_*$ of $D^b(A)$, where $V$ is a semi-homogeneous bundle of rank $r$, $f\in\mathrm{Hom}(A,A)$ an isogeny, map to the following transformations in $\mathrm{Sp}(2n,\mb Q)$:
\[
\otimes V\mapsto \begin{pmatrix}
    1 & c_1(V)/r \\ 
    0 & 1 
\end{pmatrix}, 
\quad 
f^* \mapsto \begin{pmatrix}
    f^{T} & 0 \\
    0 & f^{-1}
\end{pmatrix},
\quad 
f_* \mapsto \begin{pmatrix}
    f^{-T} & 0 \\
    0 & f
\end{pmatrix}.
\]
\end{prop}

\begin{proof}
For the functor $\otimes V$, this is \cite[Example 2.1.3]{Polishchuk_2014}. For the functors $f^*$ and $f_*$ this is clear.
\end{proof}

\section{The action of \texorpdfstring{$\mathrm{Sp}(2n,\mathbb{Q})$}{Sp(2n,Q)} on the stability manifold}\label{sec:extendspn}

In Section~\ref{sec:mirrorspn}, we considered the group of auto-equivalences $\mathrm{Aut}(D^b(E^n))$, whose quotient $ U(E^n) = \mathrm{Sp}(2n,\mathbb{Z})$, admits a well-defined action on $\Knum(E^n)$ up to sign. 
Furthermore, the short exact sequence~\eqref{eq:AutAseq} shows that $U(E^n)$ acts on the space of numerical stability conditions up to a shift ambiguity.
The purpose of this section is to extend this action to $\mathrm{Sp}(2n,\mathbb{Q})$.

Let $F \in \mathrm{End}(D^b(E^n))$ be an endo-functor and let $\sigma' = (Z', \PP') \in \mathrm{Stab}(D^b(E^n))$ be a stability condition invariant under the functors $t_a^*$ and $ ( - ) \otimes L$, with $L \in \mathrm{Pic}^0(E^n)$. We define $\sigma = F^{-1}\sigma' \coloneqq (Z,\PP)$ as follows:
\[
Z \coloneqq Z' \circ F, \quad \PP(\varphi) \coloneqq \left\{ \mc E \in D^b(E^n) \ \middle\vert\  F(\mc E) \in \PP'(\varphi) \right\}
\]
\begin{prop}\label{prop:fstability}
Assume that $F$ is one of the following:
\begin{itemize}
\item
$(-) \otimes V \colon D^b(E^n) \rightarrow D^b(E^n)$, where $V$ is a semi-homogeneous bundle.
\item 
$f^*, f_* \colon D^b(E^n) \rightarrow D^b(E^n)$, where $f$ is an isogeny.
\end{itemize}
Then $F^{-1} \sigma' \in \mathrm{Stab}(D^b(E^n))$.
\end{prop}
Consider the following subgroups and the element $J$ of $\mathrm{Sp}(2n,\mathbb{Q})$:
\begin{gather*}
N = \left\{\begin{pmatrix}
1 & A \\
0 & 1
\end{pmatrix} \ \middle\vert \ A = A^{T} \right\}, \quad
H = \left\{\begin{pmatrix}
    A & 0 \\
    0 & (A^{T})^{-1}
\end{pmatrix} \ \middle\vert \ A \in \mathrm{GL}(n,\mathbb{Q})\right\},\quad J = \begin{pmatrix}
    0 & 1 \\
    - 1 & 0 
\end{pmatrix}
\end{gather*}
and we note that by~\cite[Section 2.2]{omeara}, $\mathrm{Sp}(2n,\mathbb{Q})$ is generated by $N$, $H$, and $J$. 
Recall the isomorphism induced by the mirror map
\[
\beta \colon \NN(E^n) \otimes \mathbb{Q} \simeq H^n_{\rm pr}((E^{\circ})^n; \mathbb{Q})
\]
of Proposition~\ref{prop:HodgePrimitive}.
\begin{prop}\label{prop:spendo}
Let $\sigma \in \mathrm{Stab}(D^b(E^n))$ with central charge $Z$ and $g \in \mathrm{Sp}(2n,\mathbb{Q})$. Then there exists an endo-functor $F \in \mathrm{End}(D^b(E^n))$ such that $F^{-1}\sigma$ is a stability condition and the induced action on $\NN (E^n)$ satisfies
\[
F = \beta^{-1}\circ g_* \circ\beta.
\]
\end{prop}
\begin{rem}
We note that Proposition~\ref{prop:fstability} follows directly from \cite[Section 3.1]{Polishchuk_2014} and our exposition is similar. Nevertheless, our goal is to give a self-contained exposition without relying on the theory of Lagrangian-invariant objects described in~\cite{polishchuk2011lagrangianinvariantsheavesfunctorsabelian}.
\end{rem}
\subsection{Induced stability from endo-functors}\label{sec:endostab}
Let $X$ and $Y$ be smooth projective varieties over a field $\mathbf k$ of characteristic $0$ and consider a diagram of the form
\[
\begin{tikzcd}
D_{qcoh}(X) \arrow[r,"F", bend right = 10]&D_{qcoh}(Y)  \arrow[l, "G"', bend left = -10]\\
D^b(X) \arrow[u, hook] \arrow[r] & D^b(Y) \arrow[u,hook]
\end{tikzcd}
\]
with $D_{qcoh}(X)$ the unbounded derived category of quasi-coherent sheaves on $X$, and an adjunction $G \dashv F$. 

Throughout the rest of this section, we specialize to the case where $X = Y$ is an abelian variety, and $F$ is one of the following functors:
\begin{itemize}
\item 
$F = (-) \otimes V$, where $V \in D^b(X)$ is a semi-homogeneous bundle.
\item 
$F = f_*, f^*$, where $f \colon X \rightarrow X$ is an isogeny.
\end{itemize}

We first observe the following.

\begin{lemma}\label{lem:compzero}
The functor $F$ admits a left adjoint $G$. In particular, $G$ is also right adjoint to $F$. The compositions $GF(\EE), FG(\EE)$ can be obtained by consecutive extensions of objects of the form $t_a^*\EE \otimes L$, $L \in \mathrm{Pic}^0(X)$.
\end{lemma}
\begin{proof}
The first two sentences follow in each case from the fact that each functor is of Fourier--Mukai type, together with \cite[Proposition 5.9]{huybrechts} and the fact that the canonical bundle of $X$ is trivial.

Assume that $F = (-) \otimes V$, with $V$ a semi-homogeneous bundle on $X$. Then the left adjoint satisfies $G = (-) \otimes V^\vee$, and $GF = FG = (-) \otimes (V \otimes V^\vee)$ where $V \otimes V^\vee$ is a semi-homogeneous bundle of slope $0$. The conclusion follows from \cite[Proposition 6.15 \& Theorem 7.11(2)]{Mukai1978SemihomogeneousVB}, together with the fact that $\OO_X$ is homogeneous.

Assume that $F = f_*, f^*$. Then it suffices to show that $f_*f^*\EE$ and $f^*f_*\EE$ can be obtained by consecutive extensions of objects of the form $\EE \otimes L$, with $L \in \mathrm{Pic}^0(X)$. In the former case, we have $f_* f^* \EE = \EE \otimes f_* \OO_X$, and the conclusion follows from \cite[Corollary 4.22]{Mukai1978SemihomogeneousVB}. In the latter case, we have $f^* f_*\EE \simeq p_{1*} m^*\EE$ by~\cite[Section 1]{Oda}, where $m \colon X \times \mathrm{ker}(f) \rightarrow X$ is the restriction of the multiplication map. As the base field $\mathbf k$ is of characteristic $0$, $f$ must be a separable isogeny, and hence $p_{1*}m^* \EE \simeq t_a^*\EE$, where $a \in \mathrm{ker}(f)$.
\end{proof}
\begin{lemma}\label{lem:tstructurecond}
$F$ satisfies the following conditions:
\begin{itemize}
\item 
$F$ commutes with small coproducts.
\item 
$F(\EE) \simeq 0 \implies \EE \simeq 0 $.
\item 
$F(\EE) \in D^b(Y) \implies \EE \in D^b(X)$.
\end{itemize}
\end{lemma}
\begin{proof}
The fact that $(-) \otimes V$ and $f^*$ commute with coproducts follows from the fact that these are both left adjoint functors, together with the adjoint functor theorem. The conclusion for $f_*$ follows directly from~\cite[Lemma 1.4]{Neeman1996TheGD}.

For the second property, we note that the claim is clear for $\EE \in \mathrm{Coh}(X)$. In particular, note that all three functors are exact on the abelian category $\mathrm{Coh}(X)$, and thus commute with taking cohomology with respect to $\mathrm{Coh}(X)$. Thus, we have $F(\HH^i(\EE)) = \HH^i(F(\EE)) = 0$ for all $ i \in \mathbb{Z}$, and so $\HH^i(\EE) = 0$ for all $i \in \mathbb{Z}$, which implies that $\EE \simeq 0 $.

For the third property, let $\EE \in D_{qcoh}(X)$ such that $F(\EE) \in D^b(X)$. Then $F(\HH^i(\EE)) = \HH^i(F(\EE)) = 0$ for $i \gg 0$ and $ i \ll 0$. This implies by the above paragraph that $\HH^i(\EE) = 0$ for $i \gg 0$ and $i \ll 0$, and the conclusion follows.
\end{proof}

Following~\cite{Polishchuk_2014}, we say that a t-structure on $(\DD^{\leq 0}, \DD^{\geq 0})$ is $\mathbf{H}$-invariant, if 
\[
\EE \in \DD^{\leq 0 ( \geq 0)} \implies t_{a}^*\EE \otimes L \in \DD^{\leq 0 (\geq 0)}
\]with $ L \in \mathrm{Pic}^0(X)$. Similarly, we say that a stability condition $\sigma \in \mathrm{Stab}(D^b(X))$ is $\mathbf{H}$-invariant if $\sigma$ is invariant under the auto-equivalences $t_a^*, (-) \otimes L$, with $L \in \mathrm{Pic}^0(X)$.

\begin{corollary}[{\cite[Theorem 2.1.2]{Polishchuk2006ConstantFO}}]
Assume that $(\DD'^{\leq 0}, \DD'^{\geq 0})$ is a $\mathbf{H}$-invariant t-structure on $D^b(X)$ and define
\[
\DD^{\leq 0} \coloneqq \{ \EE\in D^b(X) \ \vert \ F(\EE) \in \DD'^{\leq 0} \}, \quad \DD^{\geq 0} \coloneqq \{ \EE\in D^b(X) \ \vert \ F(\EE) \in \DD'^{\geq 0} \}.
\]
Then $(\DD^{\leq 0}, \DD^{\geq 0})$ is an $\mathbf{H}$-invariant t-structure on $D^b(X)$. In particular $F$ and $G$ are t-exact functors.
\end{corollary}
\begin{proof}
The fact that $(\DD^{\leq 0}, \DD^{\geq 0})$ is a t-structure, and that $F$ is t-exact follows directly from~\cite[Theorem 2.1.2]{Polishchuk2006ConstantFO}. The fact that $G$ is t-exact follows from Lemma~\ref{lem:compzero}, and the assumption of $\mathbf{H}$-invariance.
\end{proof}
The following directly implies Proposition~\ref{prop:fstability} in the case of $X = Y = E^n$.
\begin{prop}[{\cite[Theorem 2.14]{Macr2007InducingSC},\ \cite{Polishchuk2006ConstantFO}}]\label{prop:inducedstab}
Let $\sigma' = (Z' ,\PP') \in \mathrm{Stab}(Y)$ be $\mathbf{H}$-invariant. Then $F^{-1}\sigma' = (Z, \PP)\in \mathrm{Stab}(X)$.
\end{prop} 
\begin{proof}

By \cite[Theorem 2.1.2]{Polishchuk2006ConstantFO}, $\AA \coloneqq \PP(0,1]$ is the heart of a bounded t-structure on $D^b(X)$. We note that the fact that $Z'$ is a stability function on $\AA' \coloneqq \PP'(0,1]$ immediately implies that $Z$ is a stability function on $\AA$. It suffices to show that the pair $(Z, \AA)$ satisfies the Harder--Narasimhan and the support property.

We claim that an object $\EE \in \AA$ is $Z$-semistable if and only if $F(\EE)$ is $Z'$-semistable in $\AA' \coloneqq \PP'(0,1]$. Assume that $A \xhookrightarrow{} \EE$ is a $Z$-destabilizing in $\AA$. As $F$ is $t$-exact and by definition of the central charge $Z$, it follows that $F(A) \xhookrightarrow{} F(\EE)$ is $Z'$-destabilizing, which is a contradiction. On the other hand, assume that $A \xhookrightarrow{} F(\EE)$ is $Z'$-destabilizing with $\EE$ being $Z$-semistable. As $G$ is $t$-exact, this induces an inclusion $G(A) \xhookrightarrow{} GF(\EE)$ which 
is $Z$-destabilizing again by definition of the central charge $Z$, and Lemma~\ref{lem:compzero}. By the assumption of $\mathbf{H}$-invariance and again by Lemma~\ref{lem:compzero}, the object $GF(\EE)$ is obtained by consecutive extensions of $Z$-semistable objects with the same slope, and hence $GF(\EE)$ must be $Z$-semistable giving a contradiction.

The existence of Harder-Narasimhan filtrations in $\AA$ with respect to $Z$ follows by the same argument as in~\cite[Theorem 2.14]{Macr2007InducingSC}. The support property follows by taking the quadratic form $Q(v) \coloneqq Q'(F(v))$. Indeed $Q(v) \geq 0$ for all $v \in \PP(\varphi)$ by the above paragraph. The fact that $Q$ is negative definite on $\mathrm{ker}(Z)$ follows by definition, unless $[F(\EE)] = 0$ with $[\EE] \neq 0$, which is impossible with our assumptions on $F$.
\end{proof}

\subsection{$\mathrm{Sp}(2n,\mathbb{Q})$ action on $\Knum(X)$}
In this section, we apply the results of Section~\ref{sec:endostab}, which implies Proposition~\ref{prop:fstability}, to deduce Proposition~\ref{prop:spendo}.

\begin{proof}[Proof of Proposition~\ref{prop:spendo}]
By Proposition~\ref{prop:endomirror} we know that every element of $N$ is induced by a functor of the form $\otimes V$ with $V$ with a semi-homogeneous bundle, and $J$ is induced by the duality functor $\Phi$. Moreover, the elements of $H$ induced by functors of the form $f_*$ with $f$ an isogeny generate $H$.
By Proposition~\ref{prop:fstability} all these endofunctors act on the space of stability conditions.
Finally, by composing the functors and the fact that $N$, $H$, and $J$ together generate $\mathrm{Sp}(2n,\mathbb{Q})$, the proposition follows.
\end{proof}

\part{A-model: An extension of the Siegel domain}\label{part:amodel}
\section{A fundamental domain in \texorpdfstring{$\mathcal{U}^+(3)$}{U+(3)}}

Throughout this section, all Lie groups are over $\mb R$ by default, e.g. $\mathrm{Sp}(2n)=\mathrm{Sp}(2n,\mb R)$.
Let $V$ be a symplectic vector space of dimension $\dim V=2n$.
The action of $\mathrm{Sp}(V)$ on ${\bigwedge}^nV^*$ and the action of $\mathrm{GL}^+(2)$ on $\mb C$ combine to give an action of $\mathrm{Sp}(V)\times \mathrm{GL}^+(2)$ on $\mb C\otimes{\bigwedge}^nV^*$ which restricts to an action on $\mc U^+(V)$.
If $n=1$ or $2$, then $\mc U^+(V)$ consists of a single orbit, but if $n=3$ it turns out there is a 3-dimensional space of orbits. The aim of this section is to describe it (Proposition~\ref{prop:u3normalform}).

A general primitive form in $\mb C\otimes{\bigwedge}^3\mb R^6$ can be written as
\begin{align}\label{form14comps}
\begin{split}
\Omega &= \alpha dz_{123} \\
&+\beta_1dz_{\ol{1}23}+\beta_2dz_{1\ol{2}3}+\beta_3dz_{12\ol{3}}+\beta_{12}(dz_{\ol{2}23}+dz_{1\ol{1}3})+\beta_{13}(dz_{\ol{3}23}+dz_{12\ol{1}})+\beta_{23}(dz_{1\ol{3}3}+dz_{12\ol{2}}) \\
&+\gamma_1dz_{1\ol{23}}+\gamma_2dz_{\ol{1}2\ol{3}}+\gamma_3dz_{\ol{12}3}+\gamma_{12}(dz_{2\ol{23}}+dz_{\ol{1}1\ol{3}})+\gamma_{13}(dz_{3\ol{23}}+dz_{\ol{12}1})+\gamma_{23}(dz_{\ol{1}3\ol{3}}+dz_{\ol{12}2}) \\
&+\delta dz_{\ol{123}}
\end{split}
\end{align}
where the 14 coefficients are in $\mb C$ and $dz_{123}=dz_1\wedge dz_2\wedge dz_3$, $dz_{\ol{1}23}=d\bar{z}_1\wedge dz_2\wedge dz_3$, etc.

\begin{lemma}\label{lem:removeoffdiagonal}
If $\Omega\in\mc U^+(3)$, then after a suitable change of coordinates in $\mathrm{Sp}(6)$, $\beta_{12}=\beta_{13}=\beta_{23}=\gamma_{12}=\gamma_{13}=\gamma_{23}=0$ in the expansion~\eqref{form14comps}, i.e. 
\begin{equation}\label{form8comp}
\Omega = \alpha dz_{123}+\beta_1dz_{\ol{1}23}+\beta_2dz_{1\ol{2}3}+\beta_3dz_{12\ol{3}}+\gamma_1dz_{1\ol{23}}+\gamma_2dz_{\ol{1}2\ol{3}}+\gamma_3dz_{\ol{12}3}+\delta dz_{\ol{123}}.
\end{equation}
\end{lemma}

\begin{proof}
According to~\cite[Corollary 2.18]{haiden20} there is a change of coordinates in $\mathrm{Sp}(6)$ such that
\[
\Omega=\mathrm{Re}(c_1dz_{123})+i\mathrm{Im}(c_2(\lambda_1dx_1+idy_1)\wedge(\lambda_2dx_2+idy_2)\wedge(\lambda_3dx_3+idy_3)).
\]
Using that $\lambda dx+idy$ is a linear combination of $dz$ and $d\bar{z}$, and expanding out, one sees that this has only terms involving $dz_{123}$, $dz_{\ol{1}23}$, $dz_{1\ol{2}3}$, $dz_{12\ol{3}}$, and their complex conjugates.
\end{proof}

The following lemma, or rather its contrapositive, provides a simple sufficient criterion for $\Omega\notin \mc U^+(3)$.

\begin{lemma}\label{lem:rootsindisk}
Suppose $\Omega\in\mc U^+(3)$ is of the form~\eqref{form8comp}. 
Then $\alpha\neq 0$ and the cubic polynomial
\[
P(z)\coloneqq\alpha z^3+(\beta_1+\beta_2+\beta_3)z^2+(\gamma_1+\gamma_2+\gamma_3)z+\delta
\]
has all roots inside the open unit disk.
In particular, $|\delta/\alpha|<1$.
\end{lemma}

\begin{proof}
This follows from the proof of~\cite[Proposition 2.5]{haiden20}. Let $w=\frac{\partial}{\partial x_1}\wedge\frac{\partial}{\partial x_2}\wedge\frac{\partial}{\partial x_3}\in {\bigwedge}^3\mb R^6$.
Then
\begin{align*}
\Omega(e^{i\theta}w)&=e^{3i\theta}\Omega^{3,0}(w)+e^{i\theta}\Omega^{2,1}(w)+e^{-i\theta}\Omega^{1,2}(w)+e^{-3i\theta}\Omega^{0,3}(w)\\
&=e^{3i\theta}\alpha+e^{i\theta}(\beta_1+\beta_2+\beta_3)+e^{-i\theta}(\gamma_1+\gamma_2+\gamma_3)+e^{-3i\theta}\delta\\
&=e^{-3i\theta}P(e^{2i\theta})
\end{align*}
Since $\Omega\in\mc U^+(3)$ and $e^{i\theta}w$ defines a path of Lagrangian subspaces in $\mb R^6$, $\Omega(e^{i\theta}w)\neq 0$ for $\theta\in\mb R$, thus $P$ has no roots on the unit circle. Also, for $\Omega=dz_{123}$, all roots are clearly inside the unit disk, so the conclusion follows from connectedness of $\mc U^+(3)$ and the fact that we can define $P$ for general $\Omega\in\mc U^+(3)$ with coefficients $\Omega^{k,3-k}(w)$ as above.
\end{proof}

\begin{lemma}\label{lem:normalize}
If $\Omega\in\mc U^+(3)$, then after a suitable change of coordinates in $\mathrm{Sp}(6)$ and multiplication by an element of $\mb C^\times$, we obtain $\Omega$ in the form~\eqref{form8comp} with $\alpha=1$ and $\gamma_1,\gamma_2,\gamma_3\in\mb [0,+\infty)$.
\end{lemma}

\begin{proof}
We may assume that $\Omega$ is already of the form~\eqref{form8comp}.
Further dividing $\Omega$ by $|\alpha|$, we may assume $|\alpha|=1$.
Acting on $\Omega$ by some element in the subgroup $U(1)^3\times U(1)\subset \mathrm{Sp}(6)\times\mb C^\times$ we can achieve $\alpha=1$ and $\gamma_1,\gamma_2,\gamma_3\in\mb [0,+\infty)$ since $U(1)^4$ acts on the lines generated by $dz_{123}$, $dz_{1\ol{23}}$, $dz_{\ol{1}2\ol{3}}$, $dz_{\ol{12}3}$ with weights
\[
\begin{pmatrix}
1 & 1 & 1 & 1 \\
1 & -1 & -1 & 1 \\
-1 & 1 & -1 & 1 \\
-1 & -1 & 1 & 1 
\end{pmatrix}
\]
which is an invertible matrix.
\end{proof}

In the following we will make use of real geometric invariant theory as developed in~\cite{richardson_slodowy} and~\cite{bohm_lafuente21}. All the statements we need can be found in~\cite[Theorem 1.1]{bohm_lafuente21}.

\begin{lemma}\label{lem:orbitsclosed}
All $\mathrm{Sp}(2n)\times\mathrm{SL}(2)$-orbits in $\mc U(n)$ are closed.
\end{lemma}

\begin{proof}
This is a strengthening of~\cite[Proposition 2.14]{haiden20} and the strategy of the proof is similar.
It suffices to check the Hilbert--Mumford criterion, i.e. closedness under the action of 1-parameter subgroups.
More precisely, let $G\coloneqq\mathrm{Sp}(2n)\times\mathrm{SL}(2)$ and $\mk g\coloneqq\mk{sp}(2n)\oplus\mk{sl}(2)$ be its Lie algebra. 
Then $\mk g$ has a Cartan decomposition $\mk g=\mk k\oplus\mk p$ where $\mk k=\mk u(n)\oplus \mk u(1)$ is the Lie algebra of the maximal compact subgroup $U(n)\times U(1)$, and $\mk p$ contains the symmetric matrices in $\mk g$.
For given $\Omega\in\mc U(n)$ and $0\neq X\in\mk p$ we need to show that if
\begin{equation}\label{eq:1grouplim}
\lim_{t\to+\infty}\Omega\exp(tX)
\end{equation}
exists, then it is contained in the same orbit.
After conjugating by some element of $G$, we may assume that $X$ is diagonal of the form
\[
\left(\mathrm{diag}(\lambda_1,\ldots,\lambda_n,-\lambda_1,\ldots,-\lambda_n),\mathrm{diag}(\mu,-\mu)\right)
\]
with $\lambda_1\geq  \lambda_2\geq \ldots\geq\lambda_n\geq 0$ and $\mu\geq 0$.

Let
\[
v\coloneqq\frac{\partial}{\partial x_1}\wedge\ldots\wedge \frac{\partial}{\partial x_n}, \qquad w\coloneqq\frac{\partial}{\partial x_1}\wedge\ldots\wedge \frac{\partial}{\partial x_{n-1}}\wedge \frac{\partial}{\partial y_n}
\]
then $a\coloneqq\Omega(v)$ and $b\coloneqq\Omega(w)$ are non-zero since $\Omega\in \mc U(n)$.
If $\Omega_t\coloneqq\Omega\exp(tX)$ then
\begin{align*}
\Omega_t(v)&=e^{(\lambda_1+\ldots+\lambda_n+\mu)t}\mathrm{Re}(a)+ie^{(\lambda_1+\ldots+\lambda_n-\mu)t}\mathrm{Im}(a) \\
\Omega_t(w)&=e^{(\lambda_1+\ldots+\lambda_{n-1}-\lambda_n+\mu)t}\mathrm{Re}(b)+ie^{(\lambda_1+\ldots+\lambda_{n-1}-\lambda_n-\mu)t}\mathrm{Im}(b).
\end{align*}
If $\mathrm{Re}(a)\neq 0$, then since $\lambda_1+\ldots+\lambda_n+\mu>0$ the limit~\eqref{eq:1grouplim} does not exist.
Assume then that $\mathrm{Re}(a)=0$ and thus $\mathrm{Im}(a)\neq 0$.
If $\lambda_1+\ldots+\lambda_n>\mu$, then again the limit does not exist, so consider the case $\lambda_1+\ldots+\lambda_n\leq\mu$.
We then have 
\[
\lambda_1+\ldots+\lambda_{n-1}-\lambda_n+\mu\geq 2\left(\lambda_1+\ldots+\lambda_{n-1}\right)
\]
which implies that the left-hand side of the inequality is positive unless $n=1$ and $\lambda_1=\mu$, in which case
\[
\lim_{t\to+\infty}\Omega_t=adx_1+\mathrm{Re}(b)dy_1
\]
which is in $\mc U(1)$ by the assumption $\Omega\in\mc U(1)$.
We may thus assume $\lambda_1+\ldots+\lambda_{n-1}-\lambda_n+\mu>0$.
If $\mathrm{Re}(b)\neq 0$, this implies that the limit~\eqref{eq:1grouplim} does not exist.
On the other hand, if $\mathrm{Re}(b)=0$, then both $a$ and $b$ are purely imaginary and thus
\[
\Omega(\cos(\theta)v+\sin(\theta)w)=a\cos(\theta)+b\sin(\theta)
\]
vanishes for some $\theta$.
Since 
\[
\cos(\theta)v+\sin(\theta)w=\frac{\partial}{\partial x_1}\wedge\ldots\wedge \frac{\partial}{\partial x_{n-1}}\wedge \left(\cos(\theta)\frac{\partial}{\partial x_n}+\sin(\theta)\frac{\partial}{\partial y_n}\right)
\]
describes a path in the Lagrangian Grassmannian, this contradicts $\Omega\in\mc U(n)$.
\end{proof}

The following lemma will be used in conjunction with Lemma~\ref{lem:rootsindisk}.

\begin{lemma}\label{lem:polyunstable}
Suppose
\[
P(z)=z^3+az^2+bz+c
\]
is a cubic polynomial with complex coefficients such that
\[
0<1-|c|^2\leq|a\ol{c}-\ol{b}|.
\]
Then $P$ has a root $r$ with $|r|\geq 1$.
\end{lemma}

\begin{proof}
We apply the Schur--Cohn test, see e.g.~\cite[Section 6.8]{henrici}.
The \textit{reciprocal adjoint polynomial} of $P$ is 
\[
P^*(z)\coloneqq \ol{c}z^3+\ol{b}z^2+\ol{a}z+1
\]
and we need to show that this has a root inside the closed unit disk.
The \textit{Schur transform}, $TP$ of $P$ is 
\[
(TP)(z)\coloneqq \ol{P(0)}P(z)-\ol{P^*(0)}P^*(z)=(a\ol{c}-\ol{b})z^2+(b\ol{c}-\ol{a})z+c\ol{c}-1.
\]
Since $(TP)(0)=c\ol{c}-1<0$ by assumption, it follows from the Schur--Cohn test that $P^*$ and $TP$ have the same number of roots inside the closed unit disk.
But if $r_1,r_2$ are the roots of $TP$, then
\[
|r_1r_2|=\frac{1-c\ol{c}}{|a\ol{c}-\ol{b}|}\leq 1
\]
so $|r_i|\leq 1$ for at least one of $i=1,2$.
\end{proof}

\begin{lemma}\label{lem:momentcondition}
$\Omega\in\mc U^+(3)$ of the type \eqref{form8comp} minimizes $\|\Omega\|^2$ on its $\mathrm{Sp}(2)^3\times\mathrm{SL}(2)$-orbit if and only if $\Omega^{2,1}=0$ and $\Omega^{0,3}=0$, i.e. the $\beta$'s and $\delta$ in the expansion \eqref{form14comps} vanish.
\end{lemma}

\begin{proof}
The Lie algebra $\mk{sl}(2)=\mk{sp}(2)$ has Cartan decomposition $\mk{sl}(2)=\mk{u}(1)\oplus\mk p$, where $\mk p$ consists of traceless symmetric matrices and has a basis
\[
X\coloneqq\begin{pmatrix} 1 & 0 \\ 0 & -1 \end{pmatrix}, \qquad Y\coloneqq\begin{pmatrix} 0 & 1 \\ 1 & 0 \end{pmatrix}.
\]
These act on 1-forms $\mb C$-linearly as 
\[
Xdz=d\bar{z},\qquad Xd\bar{z}=dz,\qquad Ydz=id\bar{z},\qquad Yd\bar{z}=-idz.
\]
The action of $X_1\coloneqq(X,0,0),Y_1\coloneqq(Y,0,0)\in \mk{sp}(2)^3$ on $\Omega$ of the type \eqref{form8comp} is given by
\begin{align*}
    X_1\Omega&=\beta_1 dz_{123}+\alpha dz_{\ol{1}23}+\gamma_3dz_{1\ol{2}3}+\gamma_2dz_{12\ol{3}}+\delta dz_{1\ol{23}}+\beta_3dz_{\ol{1}2\ol{3}}+\beta_2dz_{\ol{12}3}+\gamma_1dz_{\ol{123}} \\
    Y_1\Omega&=i\left(-\beta_1 dz_{123}+\alpha dz_{\ol{1}23}-\gamma_3dz_{1\ol{2}3}-\gamma_2dz_{12\ol{3}}-\delta dz_{1\ol{23}}+\beta_3dz_{\ol{1}2\ol{3}}+\beta_2dz_{\ol{12}3}+\gamma_1dz_{\ol{123}}\right)
\end{align*}
From this we find that
\begin{align*}
    \mathrm{Re}\langle\Omega,X_1\Omega\rangle &= \mathrm{Re}\left(\ol{\alpha}\beta_1+\ol{\beta_1}\alpha+\ol{\beta_2}\gamma_3+\ol{\beta_3}\gamma_2+\ol{\gamma_1}\delta+\ol{\gamma_2}\beta_3+\ol{\gamma_3}\beta_2+\ol{\delta}\gamma_1\right)\\
    &=2\mathrm{Re}\left(\ol{\alpha}\beta_1+\ol{\beta_2}\gamma_3+\ol{\beta_3}\gamma_2+\ol{\gamma_1}\delta\right)
\end{align*}
\begin{align*}
    \mathrm{Re}\langle\Omega,Y_1\Omega\rangle &= \mathrm{Re}\left(i\left(-\ol{\alpha}\beta_1+\ol{\beta_1}\alpha-\ol{\beta_2}\gamma_3-\ol{\beta_3}\gamma_2-\ol{\gamma_1}\delta+\ol{\gamma_2}\beta_3+\ol{\gamma_3}\beta_2+\ol{\delta}\gamma_1\right)\right)\\
    &=2\mathrm{Im}\left(\ol{\alpha}\beta_1+\ol{\beta_2}\gamma_3+\ol{\beta_3}\gamma_2+\ol{\gamma_1}\delta\right)
\end{align*}
Hence the vanishing of the variation of $\|\Omega\|^2$ along the first $\mk{sp}(2)$ factor can be expressed as
\begin{equation}
\alpha\ol{\beta_1}+\gamma_1\ol{\delta}+\beta_2\ol{\gamma_2}+\beta_3\ol{\gamma_3}=0 \label{eq:moment1}
\end{equation}
where we have taken the complex conjugate for convenience.
Similarly, the vanishing of the variation along the second a third $\mk{sp}(2)$ factor can be expressed as
\begin{align}
\alpha\ol{\beta_2}+\gamma_2\ol{\delta}+\beta_1\ol{\gamma_1}+\beta_3\ol{\gamma_3}=0 \label{eq:moment2}, \\
\alpha\ol{\beta_3}+\gamma_3\ol{\delta}+\beta_1\ol{\gamma_1}+\beta_2\ol{\gamma_2}=0 \label{eq:moment3}.
\end{align}

Let us turn to the action of $\mk{sl}(2)$ induced from its natural action on $\mb C=\mb R^2$.
Then
\begin{equation*}
    X\Omega=\ol{\Omega},\qquad Y\Omega=-i\ol{\Omega}
\end{equation*}
thus the variation of $\|\Omega\|^2$ along $\mathrm{SL}(2)$ vanishes iff $\langle\ol{\Omega},\Omega\rangle=0$, or more explicitly:
\begin{equation}
    \alpha\delta+\beta_1\gamma_1+\beta_2\gamma_2+\beta_3\gamma_3=0. \label{eq:moment4}
\end{equation}

Assume now that $\Omega$ minimizes $\|\Omega\|^2$ on its $\mathrm{Sp}(2)^3\times\mathrm{SL}(2)$-orbit, i.e. that \eqref{eq:moment1}, \eqref{eq:moment2}, \eqref{eq:moment3}, and \eqref{eq:moment4} hold.
We need to show that $\Omega^{2,1}=0$ and $\Omega^{0,3}=0$.
By Lemma~\ref{lem:normalize} we can and will assume, after a unitary change of coordinates and rescaling, that $\alpha=1$ and $\gamma_1,\gamma_2,\gamma_3\in\mb [0,+\infty)$.
Combining \eqref{eq:moment4} with \eqref{eq:moment1} (resp. \eqref{eq:moment2} or \eqref{eq:moment3}) gives
\begin{equation}
\gamma_i(\beta_i-\ol{\delta})=\ol{\beta_i}-\delta
\end{equation}
for $i=1,2,3$. Thus, for each $i$ either $\beta_i=\ol{\delta}$ or $\gamma_i=1$ (or both).
After possibly applying some permutation of $\{1,2,3\}$, there are four cases to check.

\begin{itemize}
\item \underline{Case $\beta_1=\beta_2=\beta_3=\ol{\delta}$:}
All four equations \eqref{eq:moment1}, \eqref{eq:moment2}, \eqref{eq:moment3}, \eqref{eq:moment4} become $\delta+\ol{\delta}(\gamma_1+\gamma_2+\gamma_3)=0$. 
If $\delta=0$, then $\beta_i=0$, $i=1,2,3$, and we are done.
If $\delta\neq 0$, then $\gamma_1+\gamma_2+\gamma_3=-\delta/\ol{\delta}$ has unit modulus and is non-negative, thus equal to $1$, and $\delta=\lambda i$ for some $\lambda\in[0,1)$.

The polynomial $P(z)$ as in Lemma~\ref{lem:rootsindisk} is of the form
\[
P(z)=z^3-3\lambda iz^2+z+\lambda i
\]
and since $1-\lambda^2\leq |-3\lambda^2-1|$, it follows from Lemma~\ref{lem:polyunstable} that $P$ has a root outside the open unit disk, contradicting the assumption $\Omega\in\mc U^+(3)$.

\item \underline{Case $\beta_1=\beta_2=\ol{\delta}$, $\gamma_3=1$:}
Equations~\eqref{eq:moment1} and~\eqref{eq:moment3} become
\[
\ol{\delta}(\gamma_1+\gamma_2)+\delta+\beta_3=0,\qquad \ol{\delta}(\gamma_1+\gamma_2)+\ol{\delta}+\ol{\beta_3}=0,
\]
hence, subtracting the two equations, $\delta+\beta_3\in\mb R$, and adding the two equations we see that either $\mathrm{Im}(\delta)=0$ or $\gamma_1+\gamma_2=0$.

Consider first the case $\delta\in\mb R$.
Then also $\beta_3\in\mb R$ and 
\[
\beta_1+\beta_2+\beta_3=\delta(2-\gamma_1-\gamma_2-\gamma_3)=:\delta(2-\gamma).
\]
The polynomial $P(z)$ as in Lemma~\ref{lem:rootsindisk} is then
\[
P(z)=z^3+\delta(2-\gamma)z^2+\gamma z+\delta
\]
with $\gamma\geq 1$ and $0\leq\delta<1$.
Since
\[
1-\delta^2\leq \gamma(1+\delta^2(1-2/\gamma))=|\delta^2(2-\gamma)-\gamma|,
\]
it follows from Lemma~\ref{lem:polyunstable} that $P$ has a root outside the open unit disk.

In the case $\gamma_1+\gamma_2=0$ we have $\beta_3=-\delta$, thus the polynomial $P(z)$ as in Lemma~\ref{lem:rootsindisk} is 
\[
P(z)=z^3+(2\ol{\delta}-\delta)z^2+z+\delta.
\]
Since
\[
1-\delta^2\leq |2\ol{\delta}^2-\delta\ol{\delta}-1|,
\]
it follows from Lemma~\ref{lem:polyunstable} that $P$ has a root outside the open unit disk.

\item \underline{Case $\beta_1=\ol{\delta}$, $\gamma_2=\gamma_3=1$:}
Taking the sum of \eqref{eq:moment1} and \eqref{eq:moment2} gives
\[
\delta+\ol{\delta}+\beta_2+\ol{\beta_2}+2(\gamma_1\ol{\delta}+\beta_3)=0
\]
thus $\gamma_1\mathrm{Im}(\delta)=\mathrm{Im}(\beta_3)$.
Taking the difference of \eqref{eq:moment1} and \eqref{eq:moment2}, respectively \eqref{eq:moment3} implies $\mathrm{Im}(\delta)=-\mathrm{Im}(\beta_2)=-\mathrm{Im}(\beta_3)$.
Since $\gamma_1\geq 0$, we must have $\delta,\beta_2,\beta_3\in\mb R$.

By \eqref{eq:moment4} and our assumptions we have
\[
\beta_1+\beta_2+\beta_3=-\delta\gamma_1,\qquad \gamma_1+\gamma_2+\gamma_3=\gamma_1+2.
\]
Let $\gamma\coloneqq\gamma_1$, then the polynomial $P(z)$ as in Lemma~\ref{lem:rootsindisk} is
\[
P(z)=z^3-\delta\gamma z^2+(\gamma+2)z+\delta.
\]
Since
\[
1-\delta^2\leq \left|2+\gamma(\delta^2+1)\right|,
\]
it follows from Lemma~\ref{lem:polyunstable} that $P$ has a root outside the open unit disk.

\item \underline{Case $\gamma_1=\gamma_2=\gamma_3=1$:}
The coefficient of $z$ in $P(z)$ is $\gamma_1+\gamma_2+\gamma_3=3$, so not all three roots can be inside the open unit disk.
\end{itemize}

This completes the forward implication of the lemma.
Conversely, suppose that $\Omega^{2,1}=0$ and $\Omega^{0,3}=0$. Then clearly \eqref{eq:moment1}, \eqref{eq:moment2}, \eqref{eq:moment3}, and \eqref{eq:moment4} hold, i.e. $\Omega$ is a critical point of the norm squared.
By convexity of the distance function along 1-parameter subgroups (see~\cite[Lemma 5.1]{bohm_lafuente21}), all critical points are minima, so the proof is complete.
\end{proof}

\begin{corollary}
\label{cor:momentcondition}
$\Omega\in\mc U^+(3)$ minimizes $\|\Omega\|^2$ on its $\mathrm{Sp}(6)\times\mathrm{SL}(2)$-orbit if and only if $\Omega^{2,1}=0$ and $\Omega^{0,3}=0$, i.e. the $\beta$'s and $\delta$ in the expansion \eqref{form14comps} vanish.
\end{corollary}

\begin{proof}
As before, let $\mk g=\mk{sp}(6)\oplus\mk{sl}(2)$ and $\mk g=\mk k\oplus\mk p$ its Cartan decomposition.
If $\Omega^{2,1}=\Omega^{0,3}=0$ and $X\in\mk p$, then $X\Omega$ has components of type $(1,2)$ and $(0,3)$ only, thus $\langle \Omega,X\Omega\rangle=0$, so $\Omega$ minimizes the norm squared.
For the converse, we have already seen that every $\mathrm{Sp}(6)\times\mathrm{SL}(2)$-orbit contains a minimizer with $\Omega^{2,1}=\Omega^{0,3}=0$. Since minimizers are unique up to the action of $U(3)\times U(1)$, the same is true for all minimizers.
\end{proof}

\begin{prop}\label{prop:u3normalform}
Every $\mathrm{Sp}(6)\times\mathrm{GL}^+(2)$-orbit in $\mathcal{U}^+(3)$ contains a $3$-form $\Omega$ of the type
\begin{equation}\label{form3comps}
    \Omega=dz_{123}+\gamma_1dz_{1\ol{23}}+\gamma_2dz_{\ol{1}2\ol{3}}+\gamma_3dz_{\ol{12}3}
\end{equation}
where $0\leq\gamma_1\leq\gamma_2\leq\gamma_3$ and $\gamma_1+\gamma_2+\gamma_3<1$.
Conversely, all such forms belong to $\mathcal{U}^+(3)$.
\end{prop}

\begin{proof}
Let $O\subset \mathcal{U}^+(3)$ be an $\mathrm{Sp}(6)\times\mathrm{GL}^+(2)$-orbit.
By Lemma~\ref{lem:removeoffdiagonal} we can find an $\Omega\in O$ of the form~\eqref{form8comp}.
By Lemma~\ref{lem:orbitsclosed} the $\mathrm{Sp}(2)^3\times \mathrm{SL}(2)$-orbit of $\Omega$ is closed in $\mb C\otimes {\bigwedge}^3\mb R^6$ and we can therefore assume that $\Omega$ minimizes the norm-squared along this orbit.
Then $\Omega^{2,1}=0$ and $\Omega^{0,3}=0$ by Lemma~\ref{lem:momentcondition}.
Finally, by Lemma~\ref{lem:normalize} and a possible permutation of the coordinates we can find an $\Omega$ of the type \eqref{form3comps} in $O$.
The inequality $\gamma_1+\gamma_2+\gamma_3<1$ follows directly from Lemma~\ref{lem:rootsindisk}.

To show that the forms $\Omega$ as in \eqref{form3comps} are in $\mathcal{U}^+(3)$, suppose $b_1,b_2,b_3$ is an orthonormal basis of a Lagrangian subspace in $\mb C^3$. 
Then the $b_i$ form a unitary basis of $\mb C^3$, so $|dz_{123}(b_1\wedge b_2\wedge b_3)|=1$.
By Hadamard's inequality for the determinant and the fact that the $\mb R$-linear map $\mb C^3\to \mb C^3$ with $z_1\mapsto z_1$, $z_2\mapsto \ol{z_2}$, $z_3\mapsto \ol{z_3}$ is length-preserving, we get $|dz_{1\ol{23}}(b_1\wedge b_2\wedge b_3)|\leq 1$, and similarly for $dz_{\ol{1}2\ol{3}}$ and $dz_{\ol{12}3}$.
Thus
\[
|\Omega(b_1\wedge b_2\wedge b_3)|\geq 1 - \gamma_1 - \gamma_2 - \gamma_3 > 0
\]
so $\Omega$ is non-vanishing on any Lagrangian subspace.
It is easy to see that $\Omega$ belongs to $\mathcal{U}^+(3)$ rather than $\mathcal{U}^-(3)$.
\end{proof}

\part{B-model: Construction of stability conditions}\label{part:bmodel}
\section{Product stability conditions}\label{sec:prod}

\subsection{Review of construction}

In this section, we recall constructions and results from~\cite{liu21}. We begin with the following intermediate notions of a Bridgeland stability condition. Let $X$ be a smooth projective variety, $\AA \subset D^b(X)$ be the heart of a bounded t-structure, and $Z \colon \Knum(\AA) \rightarrow \mathbb{C}$ be an additive function. Consider the following properties for the pair $(\AA,Z)$:
\begin{enumerate}
\item[(P)] 
For all $\EE \in \AA$, $\mathrm{Im}(Z(\EE)) \geq 0$. Moreover, if $\mathrm{Im}(Z(\EE)) = 0$, then $\mathrm{Re}(Z(\EE)) \leq 0$.
\item[(P$'$)] 
For all $\EE \in \AA$, $\mathrm{Im}(Z(\EE)) \geq 0$. Moreover, if $\mathrm{Im}(Z(\EE)) = 0$ and $\EE \neq 0$, then $\mathrm{Re}(Z(\EE)) < 0$.
\end{enumerate}
\begin{rem}
$Z$ is a \textit{weak stability function} if (P) is satisfied, and is a \textit{stability function} if the stronger condition (P$'$) is satisfied. 
\end{rem}
\begin{enumerate}
  \item[(HN)] $(\mathcal{A},Z)$ satisfies the Harder--Narasimhan property.
  \item[(SP)] $(\mathcal{A},Z)$ satisfies the support property.
\end{enumerate}
\begin{defn} The pair $(\mathcal{A},Z)$ is
\begin{itemize}
    \item 
 a \textit{weak pre-stability condition} if (P) + (HN) holds,
 \item 
 a \textit{weak stability condition} if (P) + (HN) + (SP) holds,
 \item 
 a \textit{stability condition} if (P$'$) + (HN) + (SP) holds.
\end{itemize}
\end{defn}

Let $S$ be a smooth projective variety and $C$ be a smooth projective curve, and consider the projections:
\[
\begin{tikzcd}
& X \coloneqq S \times C \arrow[ld,"p"']\arrow[rd,"q"]& \\
S& & C
\end{tikzcd}
\]
Fix $\sigma = (\AA,Z)$ a heart of bounded t-structure in $D^b(S)$ with a stability function $Z$.

The first step in constructing a stability condition on $X$ is the construction of a central charge function on the Abramovich--Polishchuk heart~\cite{abramovich2006sheaveststructuresvaluativecriteria} on the product variety $X$.
\begin{prop}[{\cite[Theorem 3.3]{liu21}}]
\ 
\begin{enumerate}
\item 
For any $\EE \in D^b(X)$ the function $L_{\EE}(k) \coloneqq Z\left(p_*(\EE \otimes q^* \mathcal{O}(k))\right)$ coincides with a linear polynomial for $k \gg 0$, and so we can define $a(\EE),b(\EE),c(\EE),d(\EE)\in\mb R$ by
\[
L_{\EE}(k) = a(\EE) k + b(\EE) + i (c(\EE) k + d(\EE)) \qquad \text{for }k\gg 0.
\]
\item
The following pair defines a weak pre-stability condition.
\[
\mathcal{A}_C \coloneqq \left\{ \EE \in D^b(X) \ \middle| \ p_* (\EE \otimes q^* \mathcal{O}(k)) \in \AA, k \gg 0  \right\}\]
\[
Z_C(\EE) \coloneqq \lim\limits_{k\rightarrow +\infty}\frac{Z(p_*(\EE \otimes q^* \OO(k)))}{k\cdot vol(\OO(1))} = a(\EE) + i c(\EE)
\]
\end{enumerate}
\end{prop}

As usual, in order to promote a weak stability condition to a stability condition on $X$, it is necessary to perform a tilting operation on the heart $\AA_C$. As for the case of surfaces, we first define a modified central charge and corresponding slope function depending on a parameter $t>0$:
\[
Z_t(\EE) \coloneqq a(\EE)t - d(\EE) + i c(\EE) t, \qquad \nu_t(\EE) \coloneqq \begin{cases}
  \frac{-a(\EE) t +d(\EE)}{c(\EE) t} , \quad &c(\EE) \neq 0 \\
\infty, & c(\EE) = 0
\end{cases}
\]
\begin{lemma}
The pair $(\AA_C, Z_t)$ is a weak pre-stability condition. In particular, it admits a torsion pair
\[
\AA_C = \langle \FF_t , \TT_t \rangle, \quad \mathcal{T}_t \coloneqq \left\{ \EE \in \mathcal{A}_S \ \middle\vert \ \nu_{t, min}(\EE) > 0 \right\}, \ \mathcal{F}_t \coloneqq \left\{ \EE \in \mathcal{A}_S \ \middle\vert \ \nu_{t,max}(\EE) \leq 0 \right\}
\]
where $\nu_{t,max}(\EE)$ and $\nu_{t,min}(\EE)$ are the slopes of the first and last HN factor of $\EE$, respectively.
\end{lemma}
\begin{proof}
See the paragraph above~\cite[Proposition 4.6]{liu21}.
\end{proof}

Finally, we obtain a stability condition on $X$ by the standard tilting operation.
\begin{theorem}[{\cite[Proposition 4.6]{liu21}}]\label{thm:prodstab} For $s,t > 0$, the pair $(\AA_C^t,Z_C^{s,t})$ with
\[
\AA_C^t \coloneqq \langle \TT_t, \FF_t[1] \rangle, \quad
Z_C^{s,t}(\EE) \coloneqq c(\EE) s + b(\EE) + i (-a(\EE) t + d(\EE)), 
\]
yields a stability condition on $D^b(X)$.
\end{theorem}
Let $\Lambda\coloneqq \Knum(S)$, and we note the following factorization property:
\begin{lemma}[{\cite[Lemma 5.2]{liu21}}]\label{lem:factor}
The central charge $Z_{C}^{s,t}$ factors as follows
\[
\begin{tikzcd}
\Knum(X ) \arrow[r,"\begin{pmatrix} v_1 \\ v_2\end{pmatrix}"] & \Lambda \oplus \Lambda\arrow[r,"\phi"] & \mathbb{C} 
\end{tikzcd}
\]
where $ \phi = (s \mathrm{Im}(g) - i t \mathrm{Re}(g), g)$ and $v_1, v_2:\Knum(X)\to\Lambda$ are given as follows:
\begin{align*}
v_1(\EE) &= v(p_*(\EE \otimes q^* \OO(k))) - v(p_*(\EE \otimes q^* \OO(k-1))) \\
v_2(\EE) &= v(p_*(\EE \otimes q^* \OO(k))) - kv_1(\EE).
\end{align*}
In particular, we have the equalities
\[
a(\EE) = \mathrm{Re}(g) \circ v_1, \ b(\EE) = \mathrm{Re}(g) \circ v_2, \ c(\EE) = \mathrm{Im}(g) \circ v_1, \ d(\EE) = \mathrm{Im}(g) \circ v_2.
\]
\end{lemma}
\begin{rem}\label{rem:quadratic}
$Z_{C}^{s,t}$ satisfies the support property on $\Lambda \oplus \Lambda/\ker(g)$ with respect to the quadratic form
\[
b(\EE) c(\EE) -  a(\EE) d(\EE) + \eta Q(v_1(\EE))
\]
for sufficiently small $\eta >0$. See~\cite[Lemma 5.7 \& Remark 5.8]{liu21}.
\end{rem}
\begin{rem}
One can write the factorized central charge in the suggestive form
\[
\phi = ( \tau \cdot g , g ), \quad \tau = \begin{pmatrix}
0 & s \\
-t & 0
\end{pmatrix}
\]
More generally, $\tau$ should admit values in $GL^+(2,\mathbb{R})$.
\end{rem}

Finally, we note that Theorem~\ref{thm:prodstab} implies the existence of a stability condition on the product $C_1 \times \ldots \times C_n$, of arbitrary smooth projective curves of genus $g \geq 1$. A particularly useful statement in this setting, is that stability conditions on such varieties are uniquely determined by their central charges, together with the phase of any skyscraper sheaf.

\begin{theorem}[{\cite[Theorem 6.9]{MR4469233}, \cite{LPMSZ}}]\label{thm:injective}
Let $X = C_1 \times  \ldots \times C_n$ be a product of smooth curves of positive genus. Then all (numerical) stability conditions are geometric, in other words, all
skyscraper sheaves are stable. Moreover, a stability condition is uniquely determined by its central
charge and the phase of skyscraper sheaves. Namely, the forgetful map
\begin{align*}
\mathrm{Stab}(X) &\rightarrow \mathrm{Hom}_{\mathbb{Z}}(\Knum(X), \mathbb{C}) \times \mathbb{R} \\
\sigma =(Z, \PP) & \mapsto (Z, \varphi(\OO_p))
\end{align*}
is injective.
\end{theorem}

\subsection{Central charges on products of tori}

In this section, we will take $S = E^{n-1}$ and $C = E$, where $E$ is an elliptic curve without complex multiplication. 
As before, $p$ and $q$ are the projections to the first and second factor of $X=S\times C$, respectively.  

We will explicitly compute the central charge of the product stability condition on $X$.  We take $D$ to be a divisor on $S$, and $H$ to be the class of $\OO(1)$ on $C$. We will abuse notation and use the same divisor class to denote its possible pullbacks; the ambient variety will be clear from context. The goal of this section is to establish the following:

\begin{prop}\label{prop:gcharge}
Let $\sigma \in \mathrm{Stab}(S)$ be a stability condition with central charge 
\[
Z = \begin{pmatrix} 
1+ \alpha & 0 \\
0 & 1-\alpha
\end{pmatrix} \cdot \langle \mathrm{exp}(i \beta  D),\  \cdot \ \rangle
\]
where $0 \leq \alpha < 1 $ and $\beta > 0$.
Then for $s = t = \beta$, the corresponding product stability condition has central charge:
\[
Z_C = \langle \mathrm{exp}(it (D + H))+ \alpha \, \mathrm{exp}(it (-D + H)), \ \cdot\  \rangle.
\]
\end{prop}

In order to compute the projection vectors $v_1(\EE), v_2(\EE)$ of Lemma~\ref{lem:factor}, we first note the following. 
\begin{lemma}\label{lem:grr}
For sufficiently large $k$, there is an equality
\[
\ch(p_*(\EE \otimes q^* \OO(k))) = p_* \ch(\EE \otimes q^* \OO(k)).
\]
\end{lemma}
\begin{proof}
We first note that the tangent bundle of an elliptic curve is trivial. In particular, the Todd classes of $X, S,$ and $C$ are trivial. By Grothendieck--Riemann--Roch, we have, for $k\gg 0$, an equality
\[
 p_* \ch(\EE\otimes q^* \OO(k)) = \ch(p_! (\EE \otimes q^* \OO(k))) = \ch(p_* (\EE \otimes q^* \OO(k)))
\]
which proves the claim.
\end{proof}

Let $\EE \in D^b(X)$ be an object and
\[
\ch(\EE) = (\ch_0, \ldots , \ch_n).
\]
We then have the following expansion of the Mukai pairing with the exponential of a complexified K\"ahler class $itF$. 
\begin{lemma}\label{lem:expF}
Let $F$ be a divisor on $X$. Then the Mukai pairing takes the following form:
\begin{align*}
\langle \mathrm{exp}(i t F) , \ch(\EE) \rangle = & \sum\limits_{j=0}^n \frac{(-it)^j}{j!} F^j\ch_{n-j}
\end{align*}
\end{lemma}
\begin{proof}
Evaluating directly, we obtain
\begin{align*}
\langle \mathrm{exp}(i t F) , \ch(\EE)\rangle &= \left(1, -it F, \frac{(-it)^2}{2!}F^2, \ldots, \frac{(-it)^n}{n!} F^n \right) \cdot (\ch_0, \ldots, \ch_n) \\
&=  \sum\limits_{j=0}^n \frac{(-it)^j}{j!} F^j\ch_{n-j}
\end{align*}
as desired.
\end{proof}

Next, we compute the maps $v_1, v_2$, as well as the asymptotic coefficients $a( \EE)$, $b(\EE)$, $c(\EE)$, $d(\EE)$ of Lemma~\ref{lem:factor}, which together determine the central charge $Z_C^{s,t}$. We recall the definitions:
\begin{equation}\label{eqn:v1v2}
\begin{split}
v_1(\EE) &\coloneqq v(p_*(\EE \otimes q^* \OO(k))) - v(p_*(\EE \otimes q^* \OO(k-1))) \\
v_2(\EE) &\coloneqq v(p_*(\EE \otimes q^* \OO(k))) - kv_1(\EE) \\
a(\EE) = \mathrm{Re}(g) \circ v_1, &\ b(\EE) = \mathrm{Re}(g) \circ v_2, \ c(\EE) = \mathrm{Im}(g) \circ v_1, \ d(\EE) = \mathrm{Im}(g) \circ v_2.
\end{split}
\end{equation}

\begin{corollary}\label{cor:coeffs}
The maps $v_1,v_2$ can be written as
\[
v_1(\EE) = \big( p_*(H \ch_0) , p_*(H \ch_1) , \ldots , p_*(H \ch_{n-1}) \big), \ \ 
v_2(\EE) = \big(p_*(\ch_1) , p_*( \ch_2) , \ldots , p_*(\ch_n) \big).
\]
Assume $g \coloneqq \langle \mathrm{exp}(it D) , \cdot \rangle$. Then there are equalities:
\begin{align*}
a(\EE) + i c(\EE) =H \sum\limits_{j = 1}^{n} \frac{(-it)^{j-1}}{(j-1)!} D^{j-1}\ch_{n-j},\quad b(\EE) + i d(\EE) = \sum\limits_{j = 0}^{n-1} \frac{(-it)^{j}}{j!}D^{j} \ch_{n-j}.
\end{align*}

\end{corollary}
\begin{proof}
We have
\begin{align*}
\ch(p_*( \EE \otimes q^*\OO(k))) &= p_*(\ch(\EE) \cdot \ch(q^*\OO(k))) \\
&= p_*( (\ch_0, \ldots , \ch_n) \cdot (1, kH ,0 , \ldots, 0 )) \\
&= p_*( \ch_0, kH \ch_0 + \ch_1, kH \ch_1 + \ch_2, \ldots, kH\ch_{n-1}+  \ch_{n}) \\
&= k \big( p_*(H \ch_0) , p_*(H \ch_1) , \ldots , p_*(H \ch_{n-1}) \big) + \big(p_*(\ch_1) , p_*( \ch_2) , \ldots , p_*(\ch_n) \big)
\end{align*}
which gives the desired expansion for the vectors $v_1$ and $v_2$.

From equation~(\ref{eqn:v1v2}), the asymptotic coefficients can be expressed as
\[
a(\EE) + i c(\EE) = \langle \mathrm{exp}(itD) , v_1(\EE) \rangle, \quad b(\EE) + id(\EE) = \langle \mathrm{exp}(itD) , v_2(\EE) \rangle.
\]
The claimed expansion for the asymptotic coefficients then follows directly from Lemma~\ref{lem:expF}, together with the projection formula for intersection products.
\end{proof}

\begin{lemma}\label{lem:int}
There are intersection relations
\[
D^k = 0, \ k > n-1, \quad H^k = 0, \ k > 1
\]
on the varieties $X, S$, and $C$.
\end{lemma}
\begin{proof}
This follows from the fact that the intersection of $k$ divisors on a smooth projective variety $Y$, vanishes for $k > \mathrm{dim}(Y)$, together with the fact that intersections commutes with flat pullbacks.
\end{proof}

We now combine the above results to deduce an explicit expression for the product central charge.

\begin{proof}[Proof of Proposition~\ref{prop:gcharge}]
By Theorem~\ref{thm:prodstab} and Lemma~\ref{lem:factor}, the product central charge associated with
\[
g = \begin{pmatrix}
1 + \alpha & 0 \\
0 & 1-\alpha 
\end{pmatrix} \cdot Z, \quad Z = \langle \mathrm{exp}(itD),\cdot \rangle
\]
 can be written as
\begin{align*}
Z_C(\EE) \coloneqq Z_{C_1}(\EE) + Z_{C_2}(\EE), \quad   &Z_{C_1}(\EE) \coloneqq (-it)( a(\EE) + i c(\EE)) + (b(\EE) + i d(\EE)) \\
 \quad  &Z_{C_2}(\EE) \coloneqq (-it)(  (\alpha)a(\EE) + i (-\alpha) c(\EE)) + ( (\alpha) b(\EE) + i (- \alpha) d(\EE)).
\end{align*}

By Corollary~\ref{cor:coeffs}, we have:
\begin{align*}
Z_{C_1}(\EE) = & \ch_n + \sum\limits_{j = 1}^{n-1}  \left( (-it) \frac{(-it)^{j-1}}{(j-1)!} H D^{j-1} + \frac{(-it)^{j}}{j!} D^{j}  \right)\ch_{n-j}  + (-it) \frac{(-it)^{n-1}}{(n-1)!}H \ch_0D^{n-1}\\
= & \ch_n + \sum\limits_{j = 1}^{n-1}  \left( \frac{(-it)^{j}}{j!} (H+D)^{j}  \right)\ch_{n-j}  + \frac{(-it)^{n}}{n!} (H+D)^n \ch_0
\end{align*}
which follows from the fact that
\begin{equation} \label{eqn:dhsum}
(H+D)^j = D^{j} + j HD^{j-1}
\end{equation}
for $j = 1, \ldots, n$ by Lemma~\ref{lem:int}. Applying Lemma~\ref{lem:expF}, we immediately obtain the equality
\begin{equation}\label{eqn:zc1}
Z_{C_1}(\EE) = \langle \mathrm{exp}(it (H+D)) , \cdot \rangle.
\end{equation}

To compute $Z_{C_2}(\EE)$, we note the equalities:
\begin{align*}
\alpha a(\EE) + i (-\alpha) c(\EE) &= \alpha H \sum\limits_{j = 1}^{n} (-1)^{j-1} \frac{(-it)^{j-1}}{(j-1)!} D^{j-1}\ch_{n-j},\\  \alpha b(\EE) + i(-\alpha) d(\EE) &= \alpha \sum\limits_{j = 0}^{n-1} (-1)^j\frac{(-it)^{j}}{j!}D^{j} \ch_{n-j}.    
\end{align*}
This implies the expansion:
\begin{align*}
Z_{C_2}(\EE) =&\alpha\left( \ch_n + \sum\limits_{j = 1}^{n-1}  \bigg( (-it) \frac{(it)^{j-1}}{(j-1)!} H D^{j-1} + \frac{(it)^{j}}{j!} D^{j}  \bigg)\ch_{n-j}  + (-it) \frac{(it)^{n-1}}{(n-1)!}H \ch_0D^{n-1} \right)\\
=&\alpha\left( \ch_n + \sum\limits_{j = 1}^{n-1}  \bigg( \frac{(it)^{j}}{j!} (D-H)^{j}  \bigg)\ch_{n-j}  + \frac{(it)^{n}}{n!} (D-H)^n \ch_0 \right).
\end{align*}
after applying equation~(\ref{eqn:dhsum}). Again applying Lemma~\ref{lem:expF}, we obtain the equality
\begin{equation}\label{eqn:zc2}
Z_{C_2}(\EE) = \alpha \langle \mathrm{exp}(it ( H - D)) , \cdot \rangle.
\end{equation}
Combining equations~(\ref{eqn:zc1}) and (\ref{eqn:zc2}) implies the conclusion.
\end{proof}

\subsection{The Poincar\'e Bundle}\label{subsec:poincarebundle}
In this section, we consider the auto-equivalence associated with the Poincar\'e bundle and its action on product stability conditions. 

Let $X$ be an abelian variety of dimension $n$, $\Phi_{\PP}$ be the Fourier--Mukai functor associated with the Poincar\'e bundle $\PP$ on $X \times \widehat{X}$, and $\varphi_L \colon X\rightarrow \widehat{X}$ be the morphism associated with a principal polarization $L$. 
Recall the involution operation on $X$ and the Poincar\'e duality isomorphism
\begin{align*}
i \colon X \rightarrow X ,\quad x \mapsto -x, \quad \mathrm{PD}_k \colon \mathrm{H}^k(X, \mathbb{Z}) \rightarrow \mathrm{H}^{2n-k}(\widehat{X},  \mathbb{Z}).
\end{align*}
We consider the following composition of functors
\[
\begin{tikzcd}
\Phi \colon D^b(X) \arrow[r,"\Phi_{\PP}"]& D^b(\widehat{X}) \arrow[r,"\varphi_L^*"]& D^b(X).
\end{tikzcd}
\]

\begin{lemma}[{\cite[Lemma 9.23 \& Lemma 9.29 \& (9.9)]{huybrechts}}]\label{lem:poincarefunctor}
The functors $\Phi_{\PP}$ and $\Phi$ satisfy the following properties:
\begin{enumerate}
\item 
The induced map on cohomology is 
\[
\Phi_{\PP}^H = (-1)^{\frac{k(k+1)}{2} + n}\cdot  \mathrm{PD}_k \colon \mathrm{H}^k(X,\mathbb{Q}) \rightarrow \mathrm{H}^{2n-k}(\widehat{X},\mathbb{Q}).
\]
In particular, we have
\[
(\Phi^H)^{-1}= (-1)^{\frac{k(k+3)}{2}} \cdot \left((\varphi_L^*)^H \circ \mathrm{PD}_k\right).
\]
\item 
There is a natural isomorphism of functors $\Phi \circ \Phi \simeq i^* \circ [-n]$.
\item
There is an isomorphism of line bundles $\Phi(L) \simeq L^*$.
\end{enumerate}
\end{lemma}
\begin{proof}
We only need to verify the second statement of $(1)$. From $(2)$ and the fact that $i^* \circ [-n]$ acts as $(-1)^{k+n}$ on $\mathrm{H}^k$, there is an equality of functors
\[
(-1)^{k+n} (\Phi^H)^{-1} = \Phi^H = (-1)^{\frac{k(k+1)}{2} + n}\,(\varphi_L^*)^H \circ \Phi_{\PP}^H,
\]
from which the conclusion follows.
\end{proof}

For the rest of this section, we will again fix $X \coloneqq S \times C$, where $S = E^{n-1}$ and $C = E$. We denote by $D_i\coloneqq\{x_i=0\} \subseteq X$, the divisor defined by the vanishing of the projection to the $i$'th factor of $E^n$. 

\begin{rem}
We note that $X$ admits a principal polarization with divisor class $L = D_1 + \ldots + D_n$.
\end{rem}

We first note that Lemma~\ref{lem:poincarefunctor}(1) gives a straightforward action of the functor $\Phi$ on the divisors $D_i \subseteq X$. We fix the natural basis $\{e_i \}$ for $H^1(X, \mathbb{Z})$ such that the divisor $D_i$ corresponds to the class $e_{2i} \wedge e_{2i+1} \in H^2(X, \mathbb{Z})$. 

\begin{lemma}\label{lem:pd}
    $\Phi^H(D_i) = D_1 \cdot \ldots \cdot D_{i-1} \cdot D_{i+1} \cdot \ldots \cdot D_n $ 
\end{lemma}
\begin{proof}
By definition of Poincar\'e duality, we have
\[
\mathrm{PD}(e_{2i} \wedge e_{2i+1}) = \epsilon\cdot e_1 \wedge \ldots e_{2i-1} \wedge e_{2i+2} \wedge \ldots \wedge e_{2n},
\]
with the sign given by
\begin{align*}
\epsilon &= \int (e_{2i} \wedge e_{2i+1}) \wedge (e_1 \wedge \ldots e_{2i-1} \wedge e_{2i+2} \wedge \ldots \wedge e_{2n}) \\
&=\int  e_1 \wedge \ldots \wedge e_{2n} = 1.
\end{align*}
The conclusion follows directly from the straightforward identification of the differential forms with the corresponding divisor classes.
\end{proof}

\begin{lemma}\label{lem:pdminus}
Let $A_i$ be the numerical class of an algebraic cycle in codimension $i$. Then there is an equality of intersection pairings:
\[A_i \cdot A_{n-i} = (-1)^{n}\ \Phi(A_i) \cdot \Phi(A_{n-i}).
\]
\end{lemma}
\begin{proof}
We have the following equalities:
\begin{align*}
(-1)^i A_i \cdot A_{n-i} = \langle A_i , A_{n-i} \rangle = \langle \Phi(A_i), \Phi(A_{n-i}) \rangle = (-1)^{n-i} \Phi(A_i) \cdot \Phi(A_{n-i})
\end{align*}
where the first follows from the definition of the Mukai pairing together with the fact that the Chern character induces an isomorphism:
\[
\ch \colon {\rm \Knum}(X) \xrightarrow{\sim} {\rm Ch_{num}}(X).
\]
The second equality follows from the fact that $\Phi$ is an equivalence of categories, and the third follows again from the definition of the Mukai pairing together with Lemma~\ref{lem:poincarefunctor}. The conclusion follows immediately.
\end{proof}

We fix the divisor class $L= D_1 + \ldots + D_n$, and we note the following action of $\Phi$.
\begin{lemma}\label{lem:pfix}
The action of $\Phi$ satisfies:
\[
\Phi \cdot \mathrm{exp}(iL) = i^n \cdot \mathrm{exp}(iL).
\]
\end{lemma}
\begin{proof}
    By Lemma~\ref{lem:poincarefunctor}(3), we have the equality $\Phi( \OO(L) ) = \OO(-L)$, which implies the equality of Chern characters:
    \[
    \Phi\left(1, L, \ldots, \frac{1}{(n-1)!}L^{n-1}, \frac{1}{n!} L^n \right) = \left(1, -L, \ldots, \frac{(-1)^{n-1}}{(n-1)!} L^{n-1}, \frac{(-1)^n}{n!} L^n\right).
    \]
     Combined with Lemma~\ref{lem:poincarefunctor}(1), this implies the equality of algebraic cycles
    \[
    \Phi \left( \frac{1}{k!} L^k\right) = \frac{1}{(n-k)!} (-L)^{n-k}.
    \]
This implies the following equality for all $n = 0, \ldots, n$:
\[
\Phi\left(\frac{1}{k!} (iL)^k\right) = i^k \frac{1}{(n-k)!} (-L)^{n-k} = i^n \frac{1}{(n-k)!} ( iL)^{n-k},
\]
which implies the conclusion.
\end{proof}
We consider the product stability condition as in Theorem~\ref{thm:prodstab}, with the following central charge and parameters:
\[
g = \langle \mathrm{exp}(i D) ,  \ \cdot\  \rangle, \ \  D = D_1 + \ldots + D_{n-1},\  \ s = t = 1, \ \ H = D_n
\]
The following proposition describes the action of the auto-equivalence $\Phi$ on the functions $a(\EE)$, $b(\EE)$, $c(\EE)$, $d(\EE)$ appearing in the product stability condition.
\begin{prop}\label{prop:poincarecoeffs}
 We have the relations
\begin{align*}
a(\Phi(\EE)) &= (-i)^{n-1} b(\EE), \quad b (\Phi(\EE)) = - (-i)^{n-1}a(\EE) \\ c(\Phi(\EE)) &= (-i)^{n-1}d(\EE), \quad d(\Phi(\EE)) = -(-i)^{n-1} c(\EE).
\end{align*}
\end{prop}
\begin{proof}
By Corollary~\ref{cor:coeffs}, we have the equalities
\begin{align}\label{eqn:ac}
a(\EE) + i c(\EE) = H \sum\limits_{a = 1}^{n} \frac{(-i)^{a-1}}{(a-1)!} D^{a-1}\ch_{n-a}, \quad b(\EE) + i d(\EE) = \sum\limits_{a = 0}^{n-1} \frac{(-i)^{a}}{a!}D^{a} \ch_{n-a}
\end{align}
On the other hand, by Lemma~\ref{lem:poincarefunctor}(1), we have
\[
\Phi(\ch(\EE)) = (\Phi(\ch_n), \Phi(\ch_{n-1}), \ldots, \Phi(\ch_1), \Phi(\ch_0)),
\]
which implies
\begin{equation} \begin{split}\label{eqn:acphi}
a(\Phi(\EE)) + i c(\Phi(\EE)) &=  H \sum\limits_{a = 1}^{n} \frac{(-i)^{a-1}}{(a-1)!} D^{a-1} \Phi( \ch_a) = (-1)^n \sum \limits_{a = 1}^n \frac{(-i)^{a-1}}{(a-1)!} \Phi^{-1}(H D^{a-1}) \ch_a\\
&= (-1)^n \sum \limits_{a = 1}^n \frac{(-i)^{n-a}}{(n-a)!} \Phi^{-1}(H D^{n-a}) \ch_{n-a+1}\\
b(\Phi(\EE)) + i d(\Phi(\EE)) &= \sum\limits_{a = 0}^{n-1} \frac{(-i)^{a}}{a!}D^{a} \Phi(\ch_a) = (-1)^n \sum\limits_{a = 0}^{n-1} \frac{(-i)^{a}}{a!} \Phi^{-1}(D^{a}) \ch_a  \\
&=(-1)^n \sum\limits_{a =0}^{n-1} \frac{(-i)^{n-a-1}}{(n-a-1)!} \Phi^{-1}(D^{n-a-1}) \ch_{n-a-1}
\end{split}
\end{equation}
where the first equalities follow again by Corollary~\ref{cor:coeffs}. The second equalities follow from Lemma~\ref{lem:pdminus}, and the third equalities are a relabeling of indices.

We first prove the equality 
\[
a(\Phi(\EE)) + i c(\Phi(\EE)) = (-i)^{n-1} ( b(\EE) + i d(\EE)).
\]
By equations~(\ref{eqn:ac}) and (\ref{eqn:acphi}), it suffices to show the equalities
\[
(-1)^n\frac{(-i)^{n-a}}{(n-a)!} \Phi^{-1}(H D^{n-a})  \ch_{n-a +1} = (-i)^{n-1} \frac{(-i)^{a-1}}{(a-1)!} D^{a-1} \ch_{n-a+1}
\]
for all $a = 1, \ldots, n$. 

We first claim that there exists an equality of divisor classes 
\[
\Phi^{-1}(H D^{n-a}) = (-1)^{\frac{k(k+3)}{2}}D^{a-1}, \quad k \coloneqq 2(n-a +1).
\]
We take $D = \sum\limits_{j \neq n} D_j$, and we note
\[
\left( \sum\limits_{j \neq n} D_j\right)^{n-a}\cdot H = \sum\limits_{I \subseteq \{1\ldots n-1\} } (n-a)! \ D_{I \cup \{ n \}}
\]
where $I$ runs over all subsets of $\{1, \ldots, n-1 \}$ of length $n-a$, and given a subset $J\subset\{1,\ldots,n\}$, we define the following intersection of divisor classes:
\[
 D_{J} \coloneqq \prod_{j\in J}D_j.
\]
By Lemma~\ref{lem:poincarefunctor}(1) and Lemma~\ref{lem:pd}, this implies:
\[
\Phi^{-1}\left(\left( \sum\limits_{j \neq n} D_j\right)^{n-a}\cdot H \right) = (-1)^{\frac{k(k+3)}{2}} (n-a)! \sum\limits_{J \subseteq \{1\ldots n-1\} } D_{J}, 
\]
where $J$ runs over all subsets of length $a-1$. Together with the equality
\[
 \sum\limits_{J \subseteq \{1\ldots n-1\} } D_{J} = \frac{1}{(a-1)!} D^{a-1},
\]
we obtain the claim:
\[
\frac{1}{(n-a)!} \Phi^{-1}(H \cdot D^{n-a}) =(-1)^{\frac{k(k+3)}{2}} \frac{1}{(a-1)!} D^{a-1}.
\]
To conclude, it suffices to show the equality of signs:
\[
(-i)^{n-a} (-1)^{n} (-1)^{\frac{k(k+3)}{2}} =(-i)^{n-1} (-i)^{a-1} , \quad k = 2(n-a+1),
\]
which follows from a direct computation. 

The equality
\[
b (\Phi(\EE)) + id(\Phi(\EE)) =  - (-i)^{n-1}a(\EE)  - i (-i)^{n-1}c(\EE)
\]
follows from an identical argument, and we conclude.
\end{proof}

\section{The full support property}

In this section we specialize to the case
\[
\begin{tikzcd}
& X=S\times C=E^3 \arrow[ld,"p"']\arrow[rd,"q"]& \\
S=E^2 & & C=E
\end{tikzcd}
\]
where $E$ is an elliptic curve without complex multiplication.
We apply the results of Section~\ref{sec:prod} (for $n=3$) to establish the following.
\begin{prop}\label{prop:fullsupport}
There exists a stability condition $\sigma \in \mathrm{Stab}(X)$ with central charge
\[
Z = \langle \mathrm{exp}(i(D_1 + D_2 + D_3)), \ \cdot \ \rangle \colon \Knum(X) \rightarrow \mathbb{C}
\]
satisfying the support property with respect to the lattice $\Knum(X)$ of rank $14$.
\end{prop}

\subsection{Intersection theory on $E \times E \times E$}


 
We begin by establishing a basis for $\Knum(X) \otimes \mathbb{Q}$. Consider the following subvarieties of $X$. (Recall that $x_i:E^3\to E$ is the projection to the $i$'th factor, $i=1,2,3$.)
\begin{enumerate}
\item 
$D_i \coloneqq \{x_i=0\}$
\item 
$\Delta_{ij} \coloneqq \{x_i=x_j\}$
\item 
$C_{ij} \coloneqq D_i \cap D_j$
\item 
$D_{ij} \coloneqq \Delta_{ij} \cap D_k$, where $k \neq i,j$
\end{enumerate}

\begin{lemma}\label{lem:intersections}
We have the following intersection relations among the subvarieties of $X$:
\begin{enumerate}
\item
$D_i \cdot D_3 = C_{i3}, \quad i = 1,2$
\item
$D_3 \cdot D_3 = 0$
\item
$\Delta_{i3} \cdot D_3 = C_{i3}, \quad i = 1,2$
\item
$\Delta_{12} \cdot D_3 = D_{12}$
\item
$ C_{12} \cdot D_3 = 1$ 
\item
$ C_{i3} \cdot D_3 = 0, \quad i = 1,2$ 
\item
$ D_{i3} \cdot D_3 = 1, \quad i = 1,2$
\item
$D_{12} \cdot D_3 = 0$
\end{enumerate}
\end{lemma}
\begin{proof}
We first note that two distinct subvarieties of $X$ in the above are smooth, and intersect transversally and hence their intersection can be computed directly. This implies the non-vanishing results.

For the vanishing results, we note that the translations $t_{(x_1,x_2,x_3)} \colon X \rightarrow X$ preserve the numerical cycle classes. 
The vanishing intersections follow by applying a translation on one cycle, and noting that the result has zero intersection with the second cycle.
\end{proof}

\begin{lemma} \label{lem:bases}
The following gives a basis for $\Knum(X)\otimes \mathbb{Q}$:
\begin{align*}\label{eqn:bases}
\BB \coloneqq \{&\OO_{X}, \OO_{D_1}, \OO_{D_2}, \OO_{D_3}, \OO_{\Delta_{12}}, \OO_{\Delta_{13}}, \OO_{\Delta_{23}}, \OO_{C_{12}}, \OO_{C_{13}},\OO_{C_{23}},\OO_{D_{12}} , \OO_{D_{13}} , \OO_{D_{23}} , \OO_0 	\}
\end{align*}
Moreover, the Euler pairings are given by
\[
\chi = \begin{pmatrix}
0 & 0 & 0 & 0 & 0 & 0 &0& 0 & 0 & 0 & 0 & 0 &0 & 1\\
0 & 0 & 0 & 0 & 0 & 0 &0& 0 & 0 & -1 & -1 & -1 &0 & 0\\
0 & 0 & 0 & 0 & 0 & 0 &0& 0 & -1 & 0 & -1 & 0 &-1 & 0\\
0 & 0 & 0 & 0 & 0 & 0 &0& -1 & 0 & 0 & 0 & -1 &-1 & 0\\
0 & 0 & 0 & 0 & 0 & 0 &0& 0 & -1 & -1 & 0 & -1 &-1 & 0\\
0 & 0 & 0 & 0 & 0 & 0 &0& -1 & 0 & -1 & -1 & 0 &-1 & 0\\
0 & 0 & 0 & 0 & 0 & 0 &0& -1 & -1 & 0 & -1 & -1 &0 & 0\\
0 & 0 & 0 & 1 & 0 & 1 &1& 0 & 0 & 0 & 0 & 0 &0 & 0\\
0 & 0 & 1 & 0 & 1 & 0 &1& 0 & 0 & 0 & 0 & 0 &0 & 0\\
0 & 1 & 0 & 0 & 1 & 1 &0& 0 & 0 & 0 & 0 & 0 &0 & 0\\
0 & 1 & 1 & 0 & 0 & 1 &1& 0 & 0 & 0 & 0 & 0 &0 & 0\\
0 & 1 & 0 & 1 & 1 & 0 &1& 0 & 0 & 0 & 0 & 0 &0 & 0\\
0 & 0 & 1 & 1 & 1 & 1 &0& 0 & 0 & 0 & 0 & 0 &0 & 0\\
-1 & 0 & 0 & 0 & 0 & 0 &0& 0 & 0 & 0 & 0 & 0 &0 & 0\\
\end{pmatrix}
\]
\end{lemma}
\begin{proof} 
The Euler pairing can be computed directly by applying Lemma~\ref{lem:intersections}. In particular, a straightforward calculation demonstrates that the determinant of the Euler form is non-zero. The conclusion that $\BB_1$ is a basis follows from the fact that $\dim\left(\Knum(X) \otimes \mathbb{Q}\right) = 14$ by Proposition~\ref{prop:HodgePrimitive} and the subsequent remark.
\end{proof}

\subsection{Extending the support property}

We first establish the existence of a stability condition with the same central charge as in Proposition~\ref{prop:fullsupport}, together with a quadratic form satisfying the support property with respect to a lattice of rank $7$.
\begin{lemma}\label{lem:geomcharge}
There exists a stability condition with central charge
\[
\begin{tikzcd}
Z = \langle \mathrm{exp}(i(D_1 + D_2 + D_3)), \ \cdot \ \rangle \colon \Knum(X) \arrow[rd,"v", start anchor={[xshift=10ex]},]\arrow[rr]&  & \mathbb{C} \\
 &  \Lambda \oplus \Lambda \arrow[ru]&
\end{tikzcd}
\]
satisfying the support property with respect to the quotient lattice $\Lambda \oplus \Lambda/\ker(g)$ and the quadratic form
\[
Q = b(\EE) c(\EE) - a(\EE) d(\EE) + \eta Q(v_1(\EE))
\]
for sufficiently small $\eta$.
\end{lemma}
\begin{proof}
The existence of a stability condition $(\AA,Z)\in \mathrm{Stab}(S)$ with central charge
\[
g \coloneqq \langle \mathrm{exp}(i(D_1 + D_2) , \cdot \rangle \colon \Knum(S) \rightarrow \mathbb{C}
\]
follows from \cite[Theorem 6.10]{Macri2017}. By Theorem~\ref{thm:prodstab} and Lemma~\ref{lem:factor}, there exists a stability condition $(\AA_C^t, Z_{C}^{s,t}) \in \mathrm{Stab}(X)$ with central charge
\[
Z_C^{s,t}(\EE) \coloneqq c(\EE) s + b(\EE) + i (-a(\EE) t + d(\EE)) = \begin{pmatrix}
\tau \cdot g & g
\end{pmatrix}\begin{pmatrix}
    v_1 \\ v_2
\end{pmatrix}, \quad \tau = \begin{pmatrix} 0 & s \\ -t & 0\end{pmatrix}.
\]
Applying Proposition~\ref{prop:gcharge} with the parameters $n = 3$, $s = t =1$, and $\alpha = 0$, implies that the central charge takes the claimed form.

The statement about the support property follows directly from Remark~\ref{rem:quadratic}.
\end{proof}

In order to extend the quadratic form to the full rank $14$ lattice $\Knum(X)$, we observe that it suffices to find appropriate auto-equivalences preserving the stability condition. We note that a similar idea appears already in~\cite[Lemma 3.19]{Oberdieck:2018uqa}.
\begin{lemma}\label{lem:fullsupport}
Let $\mathcal{T}$ be a triangulated category, linear over some field, and with well defined and finite rank numerical Grothendieck group $\Knum(\TT)$.
Fix a numerical stability condition $\sigma = (Z, \mathcal{A})\in \mathrm{Stab}(\mathcal{T})$. Assume the following:
\begin{enumerate}
\item
The central charge $Z$ admits factorizations through quotient lattices $\Lambda_i$, $i=1,\ldots,n$,
\[
Z \colon \Knum(\TT) \xrightarrowdbl{v_i} \Lambda_i \xrightarrow{g_i}\mathbb{C}
\]
such that $\bigcap\limits_i \ker(v_i) = \{0 \}$.
\item 
There exists quadratic forms $Q_i$ on $\Lambda_i \otimes \mathbb{R}$ satisfying the support property.
\end{enumerate}
Then $Z$ satisfies the support property on $\Knum(\TT)$ with respect to the quadratic form 
\[
Q \coloneqq \lambda_1 Q_1(v_1) + \ldots +\lambda_n Q_n(v_n)
\]
for any choices of $\lambda_i > 0$, $i=1,\ldots,n$.
\end{lemma}
\begin{proof}
Clearly $Q$ is a quadratic form. It suffices to verify the following two conditions:
\begin{enumerate}
\item
$Q(v(\EE)) \geq 0$ for any semistable object $\EE \in \TT$.
\item
$Q(v(\EE)) < 0$ for any $0\neq v(\EE) \in\ker(Z)$.
\end{enumerate}

$(1)$: We note that for any semi-stable object $\EE \in \TT$, we have $Q_i(v_i(\EE)) \geq 0 $ by assumption~$(2)$ that each $Q_i$ individually satisfies the support property with respect to $\Lambda_i$. The conclusion follows from the assumption that $\lambda_i >0$ for all $i$.

$(2)$: Let $0\neq v(\EE) \in\ker(Z)$. Then by assumption~$(2)$, we have $Q_i(v_i(v(\EE)))\leq 0$  for all $i$. By the assumption on the intersection
\[
\bigcap\limits_i \ker(v_i) = \{0\},
\]
there must exist $j$ such that the projection $v_j(v(\EE)) \in \Lambda_i$ is non-zero. Again by assumption~$(2)$ that $Q_j$ satisfies the support property, we have a strict inequality $Q_j(v_j(v(\EE))) < 0 $. The conclusion $Q(v(\EE)) < 0 $ then follows from the assumption that $\lambda_i > 0$ for all $i$.
\end{proof}

In order to apply Lemma~\ref{lem:fullsupport}, we will use two classes of auto-equivalences which preserves our central charge. 
We denote by 
\[
F \colon E^3 \rightarrow E^3, \qquad (x_1, x_2, x_3) \mapsto (x_3, x_1, x_2) 
\]
the cyclic permutation of factors. In addition, we will need to combine this action with that of the Fourier--Mukai functor associated with the Poincar\'e bundle.

We consider the functor 
\[
\Phi\coloneqq \varphi_L^* \circ \Phi_{\PP} \colon D^b(X) \longrightarrow D^b(X)
\]
 where $\PP$ is the Poincar\'e bundle on $X$, and $L$ is the principal polarization $D_1 + D_2 +D_3$. As in the previous sections, we define the divisor $ H = D_3$, and we recall the projection maps
\[
\begin{tikzcd}
\Knum(X) \arrow[r,"(v_1{,} v_2)"] & \Lambda \oplus \Lambda \arrow[r]& \mathbb{C}
\end{tikzcd}
\]
given by
\begin{align*}
v_1(\EE) &\coloneqq v(p_*(\EE \otimes q^* \OO(k))) - v(p_*(\EE \otimes q^* \OO(k-1))) \\
v_2(\EE) &\coloneqq v(p_*(\EE \otimes q^* \OO(k))) - kv_1(\EE) 
\end{align*}
appearing in Lemma~\ref{lem:geomcharge}. We demonstrate that the combination of these functors suffices for the application of Lemma~\ref{lem:fullsupport} to our situation.

\begin{lemma}\label{lem:injective} The map
\[
\left(v_1 , v_1 \circ F, v_1 \circ F^2, v_1 \circ \Phi, v_1 \circ F \circ \Phi,  v_1 \circ F^2 \circ \Phi\right)^T \colon \Knum(E^3) \rightarrow \Lambda^{\oplus 6}
\]
is injective.
\end{lemma}
\begin{proof}
    Let $\EE \in D^b(X)$, and we denote $\ch(\EE) = (\ch_0, \ch_1, \ch_2, \ch_3)$. Then by Corollary~\ref{cor:coeffs} and Lemma~\ref{lem:poincarefunctor}(1), we have
\begin{equation}\label{eqn:v1proj}
\begin{split}
  \left(  v_1 \circ F^i \right)(\EE) &= \left(p_*( H \cdot F^i(\ch_0)), p_*( H \cdot F^i(\ch_1)), p_*( H \cdot F^i(\ch_2))\right) \\
  \left(  v_1 \circ F^i \circ \Phi \right)(\EE) &= \left(p_*( H \cdot (F^i \circ\Phi)(\ch_3)),p_*( H \cdot (F^i \circ\Phi)(\ch_2)),p_*( H \cdot (F^i \circ\Phi)(\ch_1)) \right)
  \end{split}
\end{equation}
    for $i = 0,1,2$. We claim that it suffices to show:
    \[
    p_* \left(H \cdot F^i (\ch_1)\right) = 0 \implies \ch_1 = 0.
    \]
Indeed, in this case, assume that  $\left(  v_1 \circ F^i \right)(\EE) = \left(  v_1 \circ F^i \circ \Phi\right)(\EE)  = 0$. Then the second line of~\eqref{eqn:v1proj} together with the claim implies 
$\Phi(\ch_2) = 0$, thus $\ch_2 = 0$
from the fact that $\Phi$ is an auto-equivalence. The vanishing, $\ch_0 = 0$, follows from the fact that $p_*(k H) = k [S]$, and the vanishing, $\ch_3 = 0$, follows again from the fact that $\Phi$ is an auto-equivalence.

Thus, it suffices to show the implication:
\begin{equation}\label{eqn:vanish}
p_*( H \cdot \ch_1) = p_*( H \cdot F(\ch_1)) =  p_*( H \cdot F(F(\ch_1))) = 0 \implies \ch_1 = 0.
\end{equation}
We have the equalities
\[
A \cdot p_*( H \cdot F^{i}(\ch_1)) = p^*A  \cdot H \cdot F^{i}(\ch_1) = F^{-i} ( p^*A \cdot H) \cdot \ch_1
\]
for $A$ a divisor on $E \times E$. Taking $A = D_1, D_2, \Delta$, the assumption of~\eqref{eqn:vanish} implies
\[
C \cdot \ch_1 = 0, \quad C \in \{ C_{12}, C_{23}, C_{13}, D_{12}, D_{13}, D_{23} \}.
\]
Evaluating with the Euler pairing of Lemma~\ref{lem:bases}, this immediately implies that $\ch_1 = 0$, and we conclude.
\end{proof}

Finally, we apply Lemma~\ref{lem:injective} to deduce Proposition~\ref{prop:fullsupport}.

\begin{proof}[Proof of Proposition~\ref{prop:fullsupport}]
We consider the following set of auto-equivalences of $D^b(X)$:
\[
\{\Phi_i\} \coloneqq \{ \mathrm{Id}, F, F^2, \Phi, F \circ \Phi, F^2 \circ \Phi\}
\]
By Lemma~\ref{lem:pfix} and Theorem~\ref{thm:injective}, the functors $\Phi_i$ fix the stability condition $\sigma$ up to the action by an element of $\widetilde{GL^+(2,\mathbb{R})}$. 
This implies that the central charge $Z$ admits the factorizations
\[
\begin{tikzcd}
Z \colon \Knum(X) \arrow[r,twoheadrightarrow,"v \circ \Phi_i"]& \Lambda \oplus \Lambda /\ker(g) \arrow[r,"g"]& \mathbb{C}
\end{tikzcd}
\]
for all $i$. In particular, we claim that the quadratic form of Lemma~\ref{lem:geomcharge} satisfies the support property on $\Lambda \oplus \Lambda /\ker(g)$ with respect to every factorization. Indeed, assume that $\EE \in D^b(X)$ is $\sigma$-semistable. Then the equality $\Phi_i \cdot \sigma = \sigma$ implies that $\Phi_i(\EE)$ is $\sigma$-semistable. Thus, we have $Q( (v\circ \Phi_i)(\EE)) \geq 0$. The condition on negative-definiteness with respect to $\ker(g)$ is immediate by Lemma~\ref{lem:geomcharge}.

Together with Lemma~\ref{lem:injective}, the set of auto-equivalences $\{\Phi_i \}$ verifies the assumptions of Lemma~\ref{lem:fullsupport}, and the conclusion follows immediately.
\end{proof}

\section{Quadratic forms for the fundamental domain}\label{sec:quadratic}

Throughout this section, as before, $C=E$, $S=E^2$, $X=S\times C= E^3$, where $E$ is an elliptic curve without complex multiplication.
Consider the divisors $D_i=\{x_i=0\}\subset X$. We defer to Section~\ref{sec:glossary} for a summary of the data specifying the stability condition that we will need.

In this section, we examine properties of the full quadratic form on the fundamental domain. We caution that the following domain $\Pi$, which we call \textit{fundamental domain}, only becomes one in the usual sense of the term after imposing the additional condition $\alpha\leq\beta\leq\gamma$. 
For our purposes, we only need the equality: 
\[
\left(\mathrm{Sp}(6,\mathbb{R}) \times \mathrm{GL}^+(2,\mathbb{R})\right) \cdot \Pi = \UU^+(3).
\] 
Recall from Proposition~\ref{prop:u3normalform} and Section~\ref{subsec:E3corr} that this was identified with the following subset of central charges:
\begin{equation}\label{eqn:domain}
\begin{split}
\Pi \coloneqq \{ Z_{(\alpha,\beta,\gamma)} \coloneqq \,& \langle \mathrm{exp}(i(D_1 + D_2 + D_3)) + \alpha\exp(i(-D_1 - D_2 + D_3))+ \beta\exp(i(-D_1 + D_2 - D_3)) \\ & + \gamma\exp(i(D_1 - D_2 - D_3)) , \,\cdot\,\rangle\ \vert\  \alpha + \beta + \gamma < 1, \ \alpha, \beta, \gamma \geq 0  \}
\end{split}
\end{equation}
Our main result is the following:
\begin{prop}\label{prop:quadratic}
Let $Z_{(\alpha,\beta,\gamma)} \in \Pi$. Then there exist real parameters $\eta_i > 0$ and $\alpha',\beta',\gamma' >0$ such that the quadratic form
\[
Q'_{(\alpha',\beta',\gamma')} \coloneqq \alpha' (Q_1^{\eta_1}+Q_4^{\eta_4}) + \beta'( Q_2^{\eta_2} + Q_5^{\eta_5}) + \gamma' (Q_3^{\eta_3} + Q_6^{\eta_6})
\]
satisfies the following:
\begin{enumerate}
    \item 
$Q'$ satisfies the support property on $\Knum(X)$ with respect to the numerical stability condition $\sigma$ with central charge $Z_{(0,0,0)}$.
    \item 
 $Q'$   is negative definite on $\mathrm{ker}\left(Z_{(\lambda \alpha, \lambda \beta, \lambda \gamma)}\right)$ for $0 \leq \lambda \leq 1$.
\end{enumerate}
\end{prop}
The strategy of our proof proceeds in the following steps:
\begin{enumerate}
\item
In Corollary~\ref{cor:reducecharge}, we express the sum of exponentials in $Z_{(\alpha,\beta,\gamma)}$ in terms of the coefficients $a,b,c,d$ combined with the automorphisms $\Phi_i$.
\item
We then apply this in Proposition~\ref{prop:reduction}, to reduce Proposition~\ref{prop:quadratic} to a statement about quadratic forms on real variables.
\item
We conclude with Proposition~\ref{prop:linalg}, which establishes the corresponding linear algebraic statement.
\end{enumerate}

\subsection{Glossary}\label{sec:glossary}
By Proposition~\ref{prop:fullsupport}, we constructed a stability condition specified by the following data:

\begin{enumerate}
\item 
The central charge:
\[
Z = \langle \mathrm{exp}(i(D_1 + D_2 + D_3)), \ \cdot \ \rangle \colon \Knum(X) \rightarrow \mathbb{C}
\]
\item 
A corresponding stability condition on $S$ as in Theorem~\ref{thm:prodstab} and Lemma~\ref{lem:factor}:
\[
g \coloneqq \mathrm{exp}(i(D_1 + D_2)) \colon \Lambda \coloneqq \Knum(S) \rightarrow \mathbb{C}, \quad Q \coloneqq - \chi (\cdot, \cdot)
\]
\item 
A collection of auto-equivalences of $D^b(X)$:
\[
\{\Phi_i \} =\{
\mathrm{Id}, F , F^2, \Phi, F \circ \Phi, F^2 \circ \Phi\}
\]
\begin{enumerate}
\item 
$F$ denotes the permutation:
\[
F \colon E^3 \rightarrow E^3, \qquad (x_1, x_2, x_3) \mapsto (x_3, x_1, x_2) 
\]
\item 
$\Phi$ denotes the composition:
\[
\begin{tikzcd}
    \Phi \colon D^b(X) \arrow[r,"\Phi_{\PP}"]& D^b(\widehat{X}) \arrow[r,"\varphi_L^*"]& D^b(X)
\end{tikzcd},\quad L = D_1 + D_2 + D_3
\]
\end{enumerate}
We denote the induced linear automorphisms of $\Knum(X)$ by the same symbols.

\item 
A collection of factorizations:
\[
\begin{tikzcd}
Z \colon \Knum(X) \arrow[rr]\arrow[rd,"v \circ \Phi_i"',twoheadrightarrow] && \mathbb{C} \\
& \Lambda \oplus \Lambda/\ker(g) \arrow[ru]&
\end{tikzcd}
\]
\begin{enumerate}
\item
$v$ denotes the composition:
\[
\begin{tikzcd}
v \colon \Knum(X) \arrow[r,"(v_1{,}v_2)"]& \Lambda \oplus \Lambda \arrow[r,twoheadrightarrow]& \Lambda \oplus \Lambda /\ker(g)
\end{tikzcd}
\]
\item 
$v_1$ and $v_2$ denote the following:
\[
v_1(\EE) \coloneqq v(p_*(\EE \otimes q^* \OO(k))) - v(p_*(\EE \otimes q^* \OO(k-1))),
\quad
v_2(\EE) \coloneqq v(p_*(\EE \otimes q^* \OO(k))) - kv_1(\EE)\]
\end{enumerate}

\item 
A family of quadratic forms on $\Knum(X)$ satisfying the support property for all $\lambda_i >0$:
\[
\lambda_1 Q_1 + \ldots + \lambda_6 Q_6, \quad Q_i \coloneqq Q_1 \circ \Phi_i
\]
\begin{enumerate}
\item
$Q_1$ denotes the quadratic form on the lattice $\Lambda \oplus \Lambda /\ker(g)$ satisfying the support property with respect to all projections $v \circ \Phi_i$:
\[
Q^\eta_1(\EE) \coloneqq b(\EE) c(\EE) - a(\EE) d(\EE) + \eta Q(v_1(\EE) ), \ 0 < \eta < \frac{1}{C}
\]
\item 
$a(\EE),\ldots,d(\EE)$ denote the coefficients:
\[a \coloneqq \mathrm{Re}(g) \circ v_1, \ b \coloneqq \mathrm{Re}(g) \circ v_2, \ c \coloneqq \mathrm{Im}(g) \circ v_1, \ d \coloneqq \mathrm{Im}(g) \circ v_2.
\]
\item 
$C$ denotes the constant appearing in the support property on $S$:
\[
Q(v_1(\EE)) \leq C \vert Z(\EE) \vert^2 = C ( a(\EE)^2 + c(\EE)^2 )
\]
with equality if and only if $v_1(\EE) = 0$.

\end{enumerate}
\end{enumerate}

We will fix $a(\EE), \ldots , d(\EE)$ to denote the asymptotic coefficients of the product stability condition associated with the central charge $\langle\mathrm{exp}(i(D_1 + D_2 + D_3)), \, \cdot \,\rangle$.
In the following subsections, we define
\begin{equation}\label{eqn:transformcoeffs}
a_i(\EE) \coloneqq a(\Phi_i(\EE)), \ldots, d_i(\EE) \coloneqq d(\Phi_i(\EE))
\end{equation}

\subsection{Reduction step} 
In this section, we reduce Proposition~\ref{prop:quadratic} to a statement about quadratic forms on real variables.

We consider the following quadratic form 
\begin{equation}\label{eqn:pdefn}
P^{C_i}_i \coloneqq b_i c_i - a_i d_i + C_i (a_i^2 + c_i^2)
\end{equation}
and real variables $a_i, b_i, c_i, d_i$ for $ i = 1, \ldots , 6$ satisfying the conditions:
\begin{equation}\label{eqn:coeffrelations}
c_1 + b_1 = c_2 + b_2 = c_3 + b_3, \ -a_1 + d_1 = -a_2 + d_2 = -a_3 + d_3.
\end{equation}
and the conditions:
\begin{equation}\label{eqn:poincarecoeffs0}
\begin{pmatrix}
a_4 & b_4 & c_4 & d_4 \\
a_5 & b_5 & c_5 & d_5 \\
a_6 & b_6 & c_6 & d_6 \\
\end{pmatrix} =
\begin{pmatrix}
-b_1 & a_1 & - d_1 & c_1 \\ 
- b_2 & a_2 & - d_2 & c_2 \\
- b_3 & a_3 & - d_3 & c_3
\end{pmatrix}.
\end{equation}
Finally, we define the following linear function:
\begin{align}
\begin{split}\label{eqn:yfunction}
Y_{\alpha,\beta,\gamma}(a_i, b_i, c_i, d_i) &\coloneqq (c_1 + b_1) + \alpha(-c_1 + b_1) + \beta (-c_2 + b_2) + \gamma( -c_3 + b_3) \\
&\quad+ i((-a_1 + d_1) + \alpha(-a_1 - d_1 ) + \beta(-a_2 - d_2 ) + \gamma(- a_3 - d_3))
\end{split}
\end{align}

\begin{rem}
We note that $P_i^{C_i}$ replicates the quadratic forms $Q_i^{\eta_i}$; the relation~(\ref{eqn:coeffrelations}) is motivated by Lemma~\ref{lem:fixedcoeffs}, the relation~(\ref{eqn:poincarecoeffs0}) is motivated by Lemma~\ref{lem:sdualityq}, and the relation~(\ref{eqn:yfunction}) is motivated by Corollary~\ref{cor:reducecharge}.
\end{rem}
The following reduces Proposition~\ref{prop:quadratic} to a statement about quadratic forms on real variables.
\begin{prop}\label{prop:reduction}
Assume that the quadratic form
\[
P_{(\alpha',\beta',\gamma')} \coloneqq \alpha' (P_1^{C_1}+P_4^{C_4}) + \beta'( P_2^{C_2} + P_5^{C_5}) + \gamma' (P_3^{C_3} + P_6^{C_6})
\]
is negative definite on $\ker(Y_{(\alpha, \beta, \gamma)})$ for some $C_i > 0$, and real parameters $\alpha',\beta',\gamma' > 0 $. Then 
\[
\alpha' (Q_1^{\eta_1}+Q_4^{\eta_4}) + \beta'( Q_2^{\eta_2} + Q_5^{\eta_5}) + \gamma' (Q_3^{\eta_3} + Q_6^{\eta_6})
\]
is negative  definite on $\mathrm{ker}(Z_{(\alpha, \beta, \gamma)})$ for $0 < \eta_i < \frac{C_i}{C}$.
\end{prop}

We first recall that certain one-dimensional subspaces of the fundamental domain $\Pi$~(\ref{eqn:domain}) can be expressed as product stability conditions.
\begin{lemma}\label{lem:1d}
The linear functional for $0 \leq\alpha < 1$:
\[
\langle \mathrm{exp}(i(D_1 + D_2 + D_3)) + \alpha \ \mathrm{exp}(i (-D_1 - D_2 + D_3)),\,\cdot\,\rangle\in \mathrm{Hom}(\Knum(X),\mathbb{C})
\]
can be realized as the central charge of a product stability condition of Theorem~\ref{thm:prodstab} with the parameters:
\[
s = t = 1, \quad g = \begin{pmatrix}
1 + \alpha & 0 \\
0 & 1- \alpha
\end{pmatrix} \cdot \langle \mathrm{exp}(i(D_1 + D_2)), \cdot \rangle
\]
In particular, the corresponding central charge can be written in the form
\[
Z(\EE) = ((1-\alpha)c(\EE) + (1+ \alpha)b(\EE)) + i(-(1+\alpha)a(\EE) + (1- \alpha)d(\EE))
\]
\end{lemma}
\begin{proof}
The first claim follows directly from Proposition~\ref{prop:gcharge} with $n = 3$. The second claim follows by definition of the coefficients $a(\,\cdot\,), \ldots, d(\,\cdot\,)$ in Section~\ref{sec:glossary}(5d).
\end{proof}

We recall that the three vectors appearing in the description of the fundamental domain $\mc D $ in (\ref{eqn:domain}) are given by 
\[
\left\{\mathrm{exp}(i(-D_1 - D_2 + D_3)), \ \mathrm{exp}(i(-D_1 + D_2 - D_3)), \ \mathrm{exp}(i(D_1 - D_2 - D_3)) \right\}
\]
which are related by the permutation automorphism $F$. Combining this action with the asymptotic coefficients, we observe the following relations.
\begin{lemma}\label{lem:fixedcoeffs}
There exist equalities of the asymptotic coefficients in~(\ref{eqn:transformcoeffs}):
\begin{align*}
c_1 + b_1 &= c_2 + b_2 = c_3 + b_3 \\
-a_1 + d_1 &=  -a_2 + d_2 = - a_3 + d_3
\end{align*}
\end{lemma}
\begin{proof}
We have the equalities of functions:
\begin{align*}
\langle \exp(i(D_1 + D_2 + D_3)) &+ \alpha \exp(i ( -D_1 - D_2 + D_3)), -\rangle = \\
&=(c_1+b_1)+ \alpha(-c_1 + b_1) + i ((-a_1 + d_1) + \alpha(-a_1 - d_1)) \\
\langle \exp(i(D_1 + D_2 + D_3)) &+ \beta \exp(i ( -D_1 + D_2 - D_3)), - \rangle = \\
&=\langle \exp(i(D_1 + D_2 + D_3)) + \beta \exp(i ( -D_1 - D_2 + D_3)),F(-)\rangle \\
&= (c_2+b_2)+ \beta(-c_2 + b_2) + i ((-a_2 + d_2) + \beta(-a_2 - d_2))\\
\langle \exp(i(D_1 + D_2 + D_3)) &+ \gamma  \exp(i ( D_1 - D_2 - D_3)), - \rangle = \\
&=\langle \exp(i(D_1 + D_2 + D_3)) + \gamma \exp(i ( -D_1 - D_2 + D_3)),F(F( -))\rangle\\
&= (c_3+b_3)+ \gamma(-c_3 + b_3) + i ((-a_3 + d_3) + \gamma(-a_3 - d_3))
\end{align*}
where the first equality follows from~(\ref{eqn:transformcoeffs}) together with Lemma~\ref{lem:1d}, the first equalities of the second two functions follow by definition of the automorphism $F$, and the second equalities follow again from~(\ref{eqn:transformcoeffs}) and Lemma~\ref{lem:1d}.

Setting $\alpha = \beta = \gamma = 0$ and equating the left hand side of the above equations implies the identities in the statement of the Lemma.
\end{proof}

Combining the above results, we immediately deduce the following.
\begin{corollary}\label{cor:reducecharge}
Consider the complexified Mukai vector
\begin{align*}
Z_{(\alpha,\beta,\gamma)} \coloneqq \ &\langle\exp(i(D_1 + D_2 + D_3)) + \alpha \exp(i (-D_1 - D_2 + D_3))\\&+ \beta \exp(i(-D_1 + D_2 - D_3)) + \gamma   \exp(i(D_1 - D_2 - D_3)), \,- \,\rangle, \quad \alpha + \beta + \gamma <1
\end{align*}
Then this can be written in the form
\begin{align*}
Z_{(\alpha,\beta,\gamma)} &= (c_1 + b_1) + \alpha(-c_1 + b_1) + \beta (-c_2 + b_2) + \gamma( -c_3 + b_3) \\
&\quad+ i((-a_1 + d_1) + \alpha(-a_1 - d_1 ) + \beta(-a_2 - d_2 ) + \gamma(- a_3 - d_3))
\end{align*}
\end{corollary}
\begin{proof}
We have the equalities:
\begin{align*}
Z_{(\alpha,\beta,\gamma)} &= Z_{(\alpha,0,0)} + Z_{(0,\beta,0)}+ Z_{(0,0,\gamma)} - 2Z_{(0,0,0)} \\
&= ((c_1+b_1)+ \alpha(-c_1 + b_1))+ i ((-a_1 + d_1) + \alpha(-a_1 - d_1))\\
&\quad +((c_2+b_2)+ \beta(-c_2 + b_2))+ i ((-a_2 + d_2) + \beta(-a_2 - d_2))\\
&\quad +((c_3+b_3)+ \gamma(-c_3 + b_3))+ i ((-a_3 + d_3) + \gamma(-a_3 - d_3))\\
&\quad - 2( (c_1+b_1) + i (-a_1 + d_1)) \\
&= (c_1 + b_1) + \alpha(-c_1 + b_1) + \beta (-c_2 + b_2) + \gamma( -c_3 + b_3) \\
&\quad + i((-a_1 + d_1) + \alpha(-a_1 - d_1 ) + \beta(-a_2 - d_2 ) + \gamma(- a_3 - d_3)).
\end{align*}
where the first equality follows by definition of $Z_{(\alpha,\beta,\gamma)}$ together with linearity, the second follows from the proof of Lemma~\ref{lem:fixedcoeffs}, and the third follows directly from Lemma~\ref{lem:fixedcoeffs}.
\end{proof}

On the other hand, the action of the functor, $\Phi$, induced from the Poincar\'e bundle gives a particularly simple relation among the asymptotic coefficients.
\begin{lemma}\label{lem:sdualityq}
Under the auto-equivalence $\Phi$, we have the following equalities for $i = 1, 2, 3:$
\[
a_i(\Phi(\EE)) = - b_i(\EE), \ b_i (\Phi(\EE)) = a_i(\EE), \ c_i(\Phi(\EE)) = -d_i(\EE), \ d_i(\Phi(\EE)) = c_i(\EE).
\]
\end{lemma}
\begin{proof}
This follows directly from Proposition~\ref{prop:poincarecoeffs} with $n = 3$, together with the definitions of Section~\ref{sec:glossary}(5b) and the fact that the functors $F$ are auto-equivalences.
\end{proof}

Finally, we recall that the vanishing of the numerical class $[\EE]$ can be characterized by the vanishing of both the associated asymptotic coefficients, and the projection vector $v_1(\EE)$.
\begin{lemma}\label{lem:vanishing}
Assume $\EE \in \Knum(X)$ such that for all $i = 1,2,3$, there are vanishings:
\begin{align*}
v_1(\EE) = v_1(F(\EE)) = v_1(F^2(\EE)) = v_1(\Phi(\EE)) =  v_1(F \circ \Phi(\EE)) =   v_1(F \circ F \circ \Phi(\EE)) = 0
\end{align*}
Then $[\EE] = 0$.
\end{lemma}
\begin{proof}
This is a restatement of Lemma~\ref{lem:injective}, with the definitions of Section~\ref{sec:glossary}(3, 4a, 5b, \ref{eqn:transformcoeffs}).
\end{proof}

\begin{proof}[Proof of Proposition~\ref{prop:reduction}]
By Corollary~\ref{cor:reducecharge}, the central charge of Proposition~\ref{prop:quadratic} can be written in the form
\begin{align*}
Z_{(\alpha,\beta,\gamma)}(\EE) &= ((c_1 + b_1) + \alpha(-c_1 + b_1) + \beta (-c_2 + b_2) + \gamma( -c_3 + b_3))(\EE) \\
&+ i((-a_1 + d_1) + \alpha(-a_1 - d_1 ) + \beta(-a_2 - d_2 ) + \gamma(- a_3 - d_3))(\EE)
\end{align*}
Let $ \EE \in \Knum(X)$ such that $Z_{(\alpha,\beta,\gamma)}(\EE) = 0$.

We first consider the case where $a_i(\EE), \ldots d_i(\EE)$ are not all zero for $i = 1,2,3$. We obtain the inequalities:
\begin{align*}
&(\alpha' (Q_1^{\eta_1}+Q_4^{\eta_4}) + \beta'( Q_2^{\eta_2} + Q_5^{\eta_5}) + \gamma' (Q_3^{\eta_3} + Q_6^{\eta_6}))(\EE) \\
=\  & (\alpha' (b_1 c_1 - a_1 d_1 + \eta_1 Q(v_1 \circ \Phi_1)) + \alpha' (b_4 c_4 - a_4 d_4 + \eta_4 Q(v_1 \circ \Phi_4)) + \ldots )(\EE)\\
\leq\  & (\alpha' (b_1 c_1 - a_1 d_1 + \eta_1 C (a_1^2 + c_1^2)) +  \alpha' (b_4 c_4 - a_4 d_4 + \eta_4 C (a_4^2 + c_4^2)) + \ldots)(\EE) \\
=\ & (\alpha' (P_1^{C_1}+P_4^{C_4}) + \beta'( P_2^{C_2} + P_5^{C_5}) + \gamma' (P_3^{C_3} + P_6^{C_6}))(a_i(\EE), \ldots , d_i(\EE)) \eqqcolon P(a_i(\EE), \ldots , d_i(\EE))
\end{align*}
where the first follows by definition in Section~\ref{sec:glossary}(5a), the second follows by the support property of the surface stability condition as in Section~\ref{sec:glossary}(5c), and the third follows by Definition~(\ref{eqn:pdefn}). By Definition~(\ref{eqn:yfunction}), we have that
\[
Z_{(\alpha,\beta,\gamma)}(\EE) = 0 \implies Y_{(\alpha,\beta,\gamma)}(a_i(\EE), \ldots, d_i(\EE)) = 0
\]
The conclusion follows from the assumption that the quadratic form $P$ is negative definite on $\mathrm{ker}(Y_{(\alpha,\beta,\gamma)}(a_i, \ldots, d_i))$ together with Lemmas~\ref{lem:fixedcoeffs} and~\ref{lem:sdualityq}, and hence the assumption that $a_i(\EE), \ldots , d_i(\EE)$ are not all zero implies that $P(a_i(\EE), \ldots , d_i(\EE))$ is strictly negative.

Otherwise, by Lemma~\ref{lem:vanishing}, it must be the case that at least one of the lattice vectors 
\[
\{ v_1(\EE), v_1(F(\EE)), v_1(F^2(\EE)), v_1 (\Phi(\EE)), v_1( F \circ \Phi(\EE)), v_1(F^2 \circ \Phi(\EE)) \}
\]
is non-zero in the lattice $ \Knum(S)$. By the assumption that $a_1(\EE) = \ldots = d_3(\EE) = 0$, it follows by the definitions in Section~\ref{sec:glossary}(5b,  \ref{eqn:transformcoeffs}) that each non-zero vector must lie in the kernel of the surface stability condition $g \colon \Knum(S) \rightarrow \mb C$. This implies:
\begin{align*}
(\alpha' (Q_1^{\eta_1}+Q_4^{\eta_4}) &+ \beta'( Q_2^{\eta_2} + Q_5^{\eta_5}) + \gamma' (Q_3^{\eta_3} + Q_6^{\eta_6}))(\EE) \\
=\  & (\alpha' (b_1 c_1 - a_1 d_1 + \eta_1 Q(v_1 \circ \Phi_1)) + \alpha' (b_4 c_4 - a_4 d_4 + \eta_4 Q(v_1 \circ \Phi_4)) + \ldots )(\EE) \\
= \ &(\alpha'(\eta_1 Q(v_1 \circ \Phi_1)+ \eta_4 Q(v_1 \circ \Phi_4)) + \ldots) (\EE) < 0
\end{align*}
where the first follows by definition in Section~\ref{sec:glossary}(5a), the second follows from the assumption of the vanishings $a_1(\EE) = \ldots = d_3(\EE) = 0$ together with Lemma~\ref{lem:sdualityq}, and the last inequality follows from the assumption that $Q$ satisfies the support property on $\Knum(S)$ and hence is negative definite on $\ker(g) \subset \Lambda$.
\end{proof}

\subsection{Definiteness of the quadratic form}
In this section, we verify the assumption of Proposition~\ref{prop:reduction}. We consider the following quadratic form:
\begin{equation}\label{eqn:pabcdefn}
P_{(\alpha,\beta,\gamma)} \coloneqq \alpha(P_1^{C_1} + P_4^{C_4}) + \beta (P_2^{C_2} + P_5^{C_5}) + \gamma (P_3^{C_3} + P_6^{C_6})
\end{equation}
for $0 < C_i < 1$, where each $P_i$ is defined as in~(\ref{eqn:pdefn}).
\begin{prop}\label{prop:linalg}
Let $\alpha, \beta, \gamma$ be parameters satisfying:
\[
\alpha, \beta, \gamma \geq 0 , \quad \alpha + \beta + \gamma < 1,
\]
and fix constants satisfying:
\[
C\coloneqq C_1 = C_2 = C_3= C_4 = C_5 = C_6.
\]
Consider the subspaces parametrized by $\lambda$:
\[
K_{\lambda} \coloneqq \mathrm{ker}\left(Y_{(\lambda \alpha, \lambda \beta, \lambda \gamma)}\right), \quad 0 \leq \lambda \leq 1.
\]
Then:
\begin{enumerate}
    \item
If $\alpha, \beta, \gamma > 0 $, fix $C$ such that $0 < C < 1 - (\alpha + \beta + \gamma)$. Then $P_{(\alpha, \beta, \gamma)}$ is negative definite on $K_{\lambda}$ for $0 \leq \lambda \leq 1$.
\item 
If $\alpha , \beta > 0$, $\gamma = 0$, fix $C, \gamma' > 0 $ such that $0 < C < 1 - (\alpha + \beta + \gamma')$. Then $P_{(\alpha, \beta, \gamma')}$ is negative definite on $K_{\lambda}$ for $0 \leq \lambda \leq 1$.
\item 
If $\alpha >0$, $\beta = \gamma = 0$, fix $C, \beta', \gamma' > 0$ such that $0 < C < 1 - (\alpha + \beta' + \gamma')$. Then $P_{(\alpha, \beta', \gamma')}$ is negative definite on $K_{\lambda}$ for $0 \leq \lambda \leq 1$.
\end{enumerate}
\end{prop}

We first note the following:
\begin{lemma}\label{lem:lambda0}
Assume that $\alpha, \beta, \gamma \geq 0$. Then the quadratic form $P_{(\alpha,\beta,\gamma)}$ is negative semi-definite on the subspace $\ker\left(Y_{(0,0,0)}\right)$ for $ 0 \leq C_i < 1$. In particular, if $\alpha,\beta,\gamma \neq 0$, then $P_{(\alpha,\beta,\gamma)}$ is strictly negative-definite.
\end{lemma}
\begin{proof}
Setting $\alpha=\beta=\gamma=0$ in~\eqref{eqn:yfunction} yields
\[
Y_{(0,0,0)} = (c_1 + b_1) + i (-a_1 + d_1).
\]
Restricting to $\ker(Y_{(0,0,0)})$ with $b_1 = - c_1$ and $a_1 = d_1$, and using conditions~\eqref{eqn:coeffrelations} and~\eqref{eqn:poincarecoeffs0}, we find:
\begin{align*}
P_{(\alpha,\beta,\gamma)} &= \alpha(-c_1^2 - a_1^2 + C_1 (a_1^2 + c_1^2)) + \beta (-c_2^2 - a_2^2 + C_2 (a_2^2 + c_2^2)) + \gamma (-c_3^2 - a_3^2 + C_3(a_3^2 + c_3^2)) \\
&+ \alpha(-c_1^2 - a_1^2 + C_4 (b_1^2 + d_1^2)) + \beta (-c_2^2 - a_2^2 + C_5 (b_2^2 + d_2^2)) + \gamma (-c_3^2 - a_3^2 + C_6(b_3^2 + d_3^2)) \\
&= \alpha( -2 + C_1 + C_4) (a_1^2 + b_1^2) + \beta (-2 + C_2 + C_5)(a_2^2 + b_2^2) + \gamma (-2 + C_3 + C_6 ) (a_3^2 + b_3^2)
\end{align*}
which is clearly negative semi-definite with the assumptions on $\alpha,\beta,\gamma, C_i$. In particular, this is strictly negative definite with the additional assumption that $\alpha, \beta,\gamma \neq 0 $.
\end{proof}

\begin{proof}[Proof of Proposition~\ref{prop:linalg}]
We define the variables:
\[
S \coloneqq b_1 + c_1 = b_2 + c_2  = b_3 + c_3, \quad 
T \coloneqq -a_1 + d_1 = -a_2 + d_2 = -a_3 + d_3.
\]
where the equalities follow from~(\ref{eqn:coeffrelations}), and we recall the component quadratic forms~(\ref{eqn:pdefn}):
\begin{align*}
P_i^{C_i} &\coloneqq b_i c_i - a_i d_i + C_i(a_i^2 + c_i^2), \quad i = 1, 2, 3 \\
P_i^{C_i} &\coloneqq b_i c_i - a_i d_i + C_i(a_i^2 + c_i^2) = b_{i-3} c_{i-3} - a_{i-3} d_{i-3} + C_i(b_{i-3}^2 + d_{i-3}^2), \quad i = 4,5,6
\end{align*}
where the equality follows from~(\ref{eqn:poincarecoeffs0}).

This implies that we can write the full quadratic form~(\ref{eqn:pabcdefn}) as
\begin{equation}\label{eqn:pquadratic}
\begin{split}
P_{(\alpha,\beta,\gamma)} =&\  \alpha(b_1 (S - b_1) -a_1 (T+ a_1) + C_1((S - b_1)^2 + a_1^2)) \\ &+  \alpha(b_1 (S - b_1) -a_1 (T+ a_1) + C_4 (b_1^2+ (T+a_1)^2))  \\
&+ \beta(b_2 (S - b_2) - a_2 (T+ a_2) + C_2((S - b_2)^2 + a_2^2)) \\ &+ \beta(b_2 (S - b_2) - a_2 (T + a_2) + C_5 (b_2^2 + (T+a_2)^2)) \\
&+ \gamma(b_3(S- b_3) - a_3(T+a_3) + C_3((S - b_3)^2 + a_3^2)) \\ &+ \gamma(b_3 (S - b_3) - a_3(T+a_3) + C_6(b_3^2 +(T+a_3)^2))
\end{split}
\end{equation}
The linear function $Y_{(\lambda \alpha, \lambda \beta, \lambda \gamma)}$ as in Definition~(\ref{eqn:yfunction}) can be written as:
\begin{align*}
\mathrm{Re}(Y_{(\lambda \alpha, \lambda \beta, \lambda \gamma)}) &= S+ \lambda(\alpha( 2 b_1 - S) + \beta (2 b_2 - S) + \gamma ( 2 b_3 - S)) \\
\mathrm{Im}(Y_{(\lambda \alpha, \lambda \beta, \lambda \gamma)}) &= T +\lambda( \alpha (-2a_1 - T) + \beta(-2 a_2 - T) + \gamma(-2 a_3 - T))
\end{align*}
The assumptions imply that $1 - \lambda(\alpha+\beta+\gamma) \neq 0$, and so the condition $Y = 0$ gives:
\begin{equation}\label{eqn:S}
S = \frac{-2\lambda}{1 - \lambda(\alpha + \beta + \gamma)}( \alpha b_1 + \beta b_2 + \gamma b_3), \quad T = \frac{2 \lambda}{1 - \lambda(\alpha + \beta + \gamma)}(\alpha a_1 + \beta a_2 + \gamma a_3)
\end{equation}
\begin{enumerate}
\item \underline{Case $\alpha,\beta,\gamma > 0$:}
Rewriting the quadratic form $P_{(\alpha,\beta,\gamma)}$ of~(\ref{eqn:pquadratic}) using ~(\ref{eqn:S}), the quadratic form $P_{(\alpha,\beta,\gamma)}(a_i, \ldots , d_i) = P_{\lambda}(a_i, b_i)$ can be written as a function independent of the variables $c_i, d_i$. It suffices to show that for any $\lambda$ satisfying $0 \leq \lambda \leq1$ and constants $C_i$ as in the assumption, the quadratic form $P_{\lambda}(a_i, b_i)$ is negative definite.

By Lemma~\ref{lem:lambda0}, we have that $P_{\lambda = 0}(a_i, b_i)$ is negative definite, and hence it suffices to show that the difference:
\[
P_{\lambda = 0}(a_i, b_i) - P_{\lambda}(a_i, b_i) 
\]
is positive semi-definite for any values for $a_i, b_i$. We will show that this is true individually for the terms in Equation~(\ref{eqn:pquadratic}) depending only on $b_i$ and $a_i$ separately. With the variable $S$, the terms in $P_{\lambda}(a_i, b_i)$ depending only on the variables $b_i$ can be expressed as 
\begin{align*}\label{eqn:Q}
P_{\lambda} \vert_{b_i}&= \alpha(b_1 (S- b_1) + C_1 (S- b_1)^2) + \beta(b_2 (S-b_2) + C_2 (S - b_2)^2) + \gamma (b_3 (S - b_3) + C_3 (S- b_3)^2)\\
&+ \alpha(b_1 (S - b_1) + C_4 b_1^2) + \beta (b_2 (S - b_2) + C_5 b_2^2) + \gamma(b_3 (S - b_3) + C_6 b_3^2)
\end{align*}
Taking the difference and setting $C\coloneqq C_1 = \ldots = C_6$, we find
\[
P_{\lambda = 0}\vert_{b_i} - P_{\lambda}\vert_{b_i}  = - \frac{4 ( \alpha b_1 + \beta b_2 + \gamma b_3)^2 \lambda (-1 + C+ \lambda(\alpha + \beta + \gamma) )}{(-1 +\lambda (\alpha + \beta + \gamma))^2}
\]
which is clearly positive semi-definite with the assumption that $0 < C < 1 - (\alpha + \beta + \gamma)$. An identical calculation with the variable $a_i$ gives the expression
\[
P_{\lambda = 0}\vert_{a_i} - P_{\lambda}\vert_{a_i} = -\frac{4(\alpha a_1 + \beta a_2 + \gamma a_3)^2 \lambda(-1 + C + \lambda(\alpha + \beta + \gamma))}{(-1 + \lambda(\alpha + \beta + \gamma))^2}
\]
This expression together with the assumption $0 < C < 1 - (\alpha + \beta +\gamma)$ implies the result.

\item \underline{Case $\alpha,\beta >0 ,\ \gamma = 0$:}
Equation~(\ref{eqn:S}) reduces to
\[
S = \frac{-2\lambda}{1 - \lambda(\alpha + \beta)}( \alpha b_1 + \beta b_2), \quad T = \frac{2 \lambda}{1 - \lambda(\alpha + \beta )}(\alpha a_1 + \beta a_2 ).
\]
Rewriting the quadratic form $P_{(\alpha,\beta,\gamma')}$ of Equation~(\ref{eqn:pquadratic}) using the above, the quadratic form $P_{(\alpha,\beta,\gamma')}(a_i, \ldots , d_i) = P_{\lambda,\gamma'}(a_i, b_i)$ can be written as a function independent of the variables $c_i, d_i$. It suffices to show that for any $\lambda$ satisfying $0 \leq \lambda \leq1$ and constants $C_i,\gamma'$ as in the assumption, the quadratic form $P_{\lambda,\gamma'}(a_i, b_i)$ is negative definite.

A direct calculation demonstrates:
\begin{align*}
P_{\lambda=0,\gamma'= 0}(a_i, b_i) =& \alpha(-2 + C_1 + C_4) (a_1^2 + b_1^2)  + \beta (-2+ C_2 + C_5) (a_2^2 +b_2^2)
\end{align*}
which is clearly negative semi-definite. It suffices to show that the difference:
\begin{equation}\label{eqn:case2}
P_{\lambda = 0,\gamma'=0}(a_i,b_i) - P_{\lambda,\gamma'}(a_i,b_i)
\end{equation}
is positive semi-definite and that this is strictly positive definite when $a_1 = a_2 = b_1 = b_2 = 0$. Restricting to the variables depending only on $b_i$, and setting $C\coloneqq C_1 = \ldots = C_6$, we find
\begin{align*}
P_{\lambda = 0,\gamma'=0}\vert_{b_i} - P_{\lambda,\gamma'}\vert_{b_i} =& - 2b_3^2 (-1 + C) \gamma' + \frac{4 b_3 (-1 + C) (b_1 \alpha + b_2 \beta) \gamma' \lambda }{-1 + (\alpha + \beta) \lambda} \\
&- \frac{4 (b_1 \alpha + b_2 \beta)^2 \lambda ( -1 + C + (\alpha + \beta+ C \gamma')\lambda}{(-1 + (\alpha + \beta)\lambda)^2}
\end{align*}
With $ b_1 = b_2 = 0$, this expression reduces to $(- 2b_3^2 (-1 + C) \gamma')$, and together with the assumptions on $C_i, \gamma'$, this is strictly positive definite. Thus, it suffices to show that this difference is positive semi-definite for all $b_i$.

By Sylvester's criterion, it suffices to check that all the principal minors of the associated symmetric matrix are non-negative. Up to exchanging $\alpha \leftrightarrow \beta$, the determinants of the principal minors are:
\begin{gather*}
-4 (-1 + C) \gamma', \qquad  -\frac{8 \alpha^2 \lambda (-1 + C + (\alpha + \beta + C \gamma') \lambda)}{(-1 + (\alpha + \beta)\lambda)^2}, \\  \frac{16(-1 + C)\alpha^2 \gamma' \lambda (-2 + (2 (\alpha + \beta) + \gamma')\lambda + C(2 + \gamma' \lambda))}{(-1 + (\alpha + \beta)\lambda)^2}.
\end{gather*}
We claim that the assumptions:
\[
0 < \alpha, \beta, \gamma',C, \quad 0 < \alpha + \beta + \gamma' + C < 1, \quad 0 \leq \lambda \leq 1
\]
imply that the above determinants must be non-negative.

The conclusion for the first term is clear, as the assumptions imply $-1 + C < 0, \gamma' > 0$. For the second term, we note 
\begin{align*}
1 - C - (\alpha + \beta + C \gamma' )\lambda &> 1 - (1 - (\alpha + \beta + \gamma')) - (\alpha + \beta + C \gamma') \lambda \\
&= (\alpha + \beta + \gamma') - (\alpha + \beta + C \gamma' ) \lambda \\
& > 0
\end{align*}
where the first inequality follows from the assumption $C < 1- (\alpha + \beta + \gamma')$, and the third inequality follows from the facts that $0 < C  < 1$ and $ 0 \leq \lambda \leq 1$. For the third term, we note 
\begin{align*}
(-2 + (2 (\alpha + \beta) + \gamma')\lambda + C(2 + \gamma' \lambda)) &\leq -2 + (2 ( \alpha + \beta) + \gamma') + C ( 2 + \gamma')) \\
& = -2 + 2 ( \alpha + \beta + C) + (1 + C) \gamma' \\
& < -2 + 2( \alpha + \beta + \gamma' + C)\\ &< 0 
\end{align*}
where the first inequality follows from the assumption $0 \leq \lambda \leq 1$, the third follows from the fact that $1 < 1 + C_1 < 2 $, and the fourth follows from $0< \alpha + \beta + \gamma' + C < 1$. 

An identical calculation with Equation~(\ref{eqn:case2}) restricted to the variables $a_i$, implies the result.

\item \underline{Case $\alpha> 0,\ \beta=\gamma = 0$:}
Equation~(\ref{eqn:S}) reduces to:
\[
S = \frac{-2\lambda}{1 - \lambda\alpha }( \alpha b_1 ), \quad T = \frac{2 \lambda}{1 - \lambda\alpha }(\alpha a_1 ).
\]
Rewriting the quadratic form $P_{(\alpha,\beta',\gamma')}$ of Equation~(\ref{eqn:pquadratic}) using the above, the quadratic form $P_{(\alpha,\beta',\gamma')}(a_i, \ldots , d_i) = P_{\lambda,\beta',\gamma'}(a_i, b_i)$ can be written as a function independent of the variables $c_i, d_i$. It suffices to show that for any $\lambda$ satisfying $0 \leq \lambda \leq1$ and constants $C,\beta',\gamma'$ as in the assumption, the quadratic form $P_{\lambda,\beta',\gamma'}(a_i, b_i)$ is negative definite.

A direct calculation demonstrates:
\begin{equation}\label{eqn:case3}
P_{\lambda=0,\beta'=0,\gamma'= 0}(a_i, b_i) = 2\alpha(-1 +C) (a_1^2 + b_1^2)  
\end{equation}
which is clearly negative semi-definite. It suffices to show that the difference:
\[
P_{\lambda = 0,\beta'=0,\gamma'=0}(a_i,b_i) - P_{\lambda,\beta',\gamma'}(a_i,b_i)
\]
is strictly positive definite for $0 < \lambda \leq 1$. Indeed for $\lambda = 0$, the quadratic form $P_{\lambda = 0, \beta', \gamma'}(a_i,b_i)$ is clearly strictly negative-definite by Lemma~\ref{lem:lambda0}.

Restricting to the variables depending only on $b_i$ and applying Sylvester's criterion to the resulting quadratic form, the leading principal minors can be chosen to be:
\begin{gather*}
-4 (-1 +C) \beta', \qquad 16(-1 +C)^2 \beta' \gamma', \\ -\frac{64 (-1 + C)^2 \alpha^2 \beta' \gamma' \lambda ( 2 (-1 + C) + 2 \alpha \lambda + (1 +C) (\beta' + \gamma') \lambda)}{(-1 + \alpha \lambda)^2}    
\end{gather*}

We claim that the assumptions:
\[
0 < \alpha, \beta', \gamma',C, \quad 0 < \alpha + \beta' + \gamma' + C < 1, \quad 0 < \lambda \leq 1
\]
imply that the above determinants must strictly positive. The conclusion for the first and second terms are clear, as the assumptions imply $0 < \beta', \gamma'$, and $-1 + C < 0$. For the third term, we note
\begin{align*}
2 (-1 + C) + 2 \alpha \lambda + (1 + C)(\beta' + \gamma') \lambda &= -2 + 2(C + \alpha \lambda) + (1+ C) (\beta' + \gamma') \lambda \\
&< -2 + 2(C + \lambda(\alpha + \beta' + \gamma')) \\
&< 0
\end{align*}
where the second inequality follows from the fact that $ 1 < 1 + C < 2$, and the third follows from the fact that $ 0 < \lambda \leq 1$ and $ 0< \alpha + \beta' + \gamma' < 1-C $.

An identical calculation with Equation~(\ref{eqn:case3}) restricted to the variables $a_i$, implies the result.
\end{enumerate}
\end{proof}

We conclude this section with the proof of Proposition~\ref{prop:quadratic}.

\begin{proof}[Proof of Proposition~\ref{prop:quadratic}]
Proposition~\ref{prop:reduction} combined with Proposition~\ref{prop:linalg} implies that there exists such a quadratic form $Q'$ satisfying the second property. The fact that $Q'$ satisfies the support property on $\Knum(X)$ with respect to $\sigma$, and hence the first property, follows directly from the proof of Proposition~\ref{prop:fullsupport}.
\end{proof}
\part{Proof of the main theorem}\label{part:proof}
\section{Density Results}\label{sec:density}

\subsection{Dense group actions}
Let $G$ be a topological group acting continuously on a topological space $X$ by
\begin{align*}
\pi \colon G \times X &\rightarrow X \\
(g,x) &\mapsto g\cdot x.
\end{align*}
\begin{lemma}\label{lem:density}
Let $U \subseteq X$ be an open subset, and let $(g_0, x_0) \in G \times X$ be an element such that $g_0 \cdot x_0 \in U$. Then there exists an open neighborhood $H \subseteq G$ of $g_0$, such that $h \cdot x_0 \in U$ for all $h \in H$.
\end{lemma}
\begin{proof}
As the group action is continuous, the subset 
\[
\pi^{-1}(U) = \{(g, x) \ \vert\  g \cdot x \in U \} \subseteq G \times X
\]
is open in the product topology. In particular, there exists a basic open neighborhood, $H \times V \subseteq \pi^{-1}(U)$, of $(g_0, x_0)$ where $H \subseteq G$ and $V \subseteq X$ are both open subsets. The open subset, $H$, then gives the desired open neighborhood.
\end{proof}

\begin{corollary}\label{cor:densityfund}
Let $G' \subseteq G$ be a dense subgroup and $\UU \subseteq X$ an open subset such that $G \cdot \UU = X$. Then there is an equality of sets $G' \cdot \UU = X$. 
\end{corollary}
\begin{proof}
By the assumption that $G \cdot \UU = X$, for any element $x \in X$, there exists $g \in G$ such that $g\cdot x \in \UU$. By Lemma~\ref{lem:density}, there exists an open neighborhood, $H \subseteq G$, of $g$ such that $h \cdot x \in \UU$ for all $h \in H$. As $G' \subseteq G$ is dense, there exists an element $g' \in G'$ with $g' \in H$, from which the conclusion follows.
\end{proof}

\subsection{Density of $\mathrm{Sp}(2n,\mathbb{Q})$}
In the following, we specialize to the case of $G=\mathrm{Sp}(2n,\mathbb{R})$ and consider the subgroup $\mathrm{Sp}(2n,\mathbb{Q})$.

\begin{lemma}\label{lem:spqdense}
The subgroup $\mathrm{Sp}(2n,\mathbb{Q}) \subseteq \mathrm{Sp}(2n, \mathbb{R})$ is dense.
\end{lemma}
\begin{proof}
This follows from~\cite[Appendix A]{PRA}.
\end{proof}

Consider the following subgroups and the element $J$ of $\mathrm{Sp}(2n,\mb R)$:
\begin{gather*}
N = \left\{\begin{pmatrix}
1 & A \\
0 & 1
\end{pmatrix} \ \middle\vert \ A = A^{T} \right\}, \qquad
H = \left\{\begin{pmatrix}
    A & 0 \\
    0 & (A^{T})^{-1}
\end{pmatrix} \ \middle\vert \ A \in \mathrm{GL}(n,\mathbb{R})\right\},\quad J = \begin{pmatrix}
    0 & 1 \\
    - 1 & 0 
\end{pmatrix}
\end{gather*}

\begin{prop}\label{prop:densesub}
The group $\mathrm{Sp}(2n,\mathbb{Q})$ is generated by $H(\mathbb{Q})$, $N(\mathbb{Q})$, and $J$, and is dense in $\mathrm{Sp}(2n,\mathbb{R})$. 
Thus, if $\mathrm{Sp}(2n,\mb R)$ acts continuously on a space $X$ and $\UU \subseteq X$ is an open subset such that $\mathrm{Sp}(2n,\mathbb{R}) \cdot \UU = X$, then also $\mathrm{Sp}(2n,\mathbb{Q})  \cdot \UU = X$.
\end{prop}
\begin{proof}
The fact that $\mathrm{Sp}(2n,\mathbb{Q})$ is generated by the subgroups $H(\mathbb{Q})$, $N(\mathbb{Q})$, and $J$ follows directly from~\cite[Section 2.2]{omeara}. The density of the embedding $\mathrm{Sp}(2n,\mathbb{Q}) \subseteq \mathrm{Sp}(2n,\mathbb{R})$ from Lemma~\ref{lem:spqdense} together with Corollary~\ref{cor:densityfund} implies the second claim.
\end{proof}

\section{The embedding}

In this section, we complete the proof of Theorem~\ref{thm:main}, restated as follows:
\begin{reptheorem}{thm:main}
There is a commutative diagram
\[
\begin{tikzcd}
\widetilde{\mathcal{U}^{+}(3)} \arrow[r, hook, "i"] \arrow[d, "\pi"] & \mathrm{Stab}_{\Lambda}(X) \arrow[d,"p"]\\
\mathcal{U}^+(3) \arrow[r, hook, "j"] & \mathrm{Hom}(\Lambda, \mathbb{C})
\end{tikzcd}
\]
where $i$ is a continuous embedding onto a connected component.
\end{reptheorem}

Our goal in the subsequent sections will be to construct the map $i$, and to verify that it is indeed a continuous embedding onto a connected component.

\subsection{Preliminaries}

In the following, $X$ is a smooth projective variety and $\Lambda=\Knum(X)$ is its numerical Grothendieck group.
\begin{defn}
Given a central charge $Z$ with quadratic form $Q$ on $\Lambda\otimes\mb R$ such that $Q|_{\mathrm{ker}(Z)}<0$, we define $\PP_Z \subseteq \mathrm{Hom}(\Lambda,\mathbb{C})$ to be the connected component of the set
\[
\left\{ Z' \in \mathrm{Hom}(\Lambda, \mathbb{C}) \mid Q \vert_{\mathrm{ker}(Z')} < 0 \right\}
\]
containing $Z$. Given a phase $\varphi \in [0, 2\pi)$, we define the subset
\[(\PP_Z)_{\varphi} \coloneqq \left\{ Z' \in \PP_Z \mid \mathrm{arg}(Z'(\mc O_p)) = \varphi\right\}.\]
where $\mc O_p$ is a skyscraper sheaf on $X$.
\end{defn}

\begin{rem}
    We note that the subsets $(\PP_Z)_{\varphi}$ associated with different phases $\varphi$ are all homeomorphic to each other.
\end{rem}

We recall the following restatement of \cite[Theorem 1.2]{MR4023385}. 
We note that the explicit description of the covering map follows directly from the corresponding proof, which constructs an explicit section  $\PP_Z \rightarrow \mathrm{Stab}_{\Lambda}(X)$. Finally, we recall that the universal covering $\widetilde{\mathrm{GL}^+(2,\mathbb{R})} \rightarrow \mathrm{GL}^+(2,\mathbb{R})$ can be explicitly described as the set of pairs $(T,f)$ where $f \colon \mathbb{R} \rightarrow \mathbb{R}$ is an increasing map with $f(\varphi + 1 ) = f(\varphi) + 1$ and $T \colon \mathbb{R}^2 \rightarrow \mathbb{R}^2$ is an orientation preserving linear isomorphism such that the induced maps on $\mathbb{R}/2\mathbb{Z}$ agree. Given a phase $\varphi \in [0,1)$ and a stability condition $\sigma \in \mathrm{Stab}(X)$, we will denote by $e^{2\pi i \varphi} \cdot \sigma$ the stability condition induced via the natural action of the element $(e^{2\pi i \varphi}, f) \in \widetilde{\mathrm{GL}^+(2,\mathbb{R})}$, with $f$ uniquely specified by the condition $f(0) = \varphi$.

\begin{theorem}[{\cite[Theorem 1.2]{MR4023385}}]\label{thm:bayerdef}
Let $\sigma \in \mathrm{Stab}_{\Lambda}(X)$ be a stability condition with central charge $Z$ satisfying the support property with respect to the quadratic form $Q$ on $\Lambda \otimes \mathbb{R}$.

\begin{enumerate}
    \item 
There exists a commutative diagram 
\[
\begin{tikzcd}
U_{\sigma} \coloneqq (\PP_Z)_{\varphi=0} \times \mathbb{R} \arrow[r, hook, "i"] \arrow[d, "p_{\sigma}"]&\mathrm{Stab}_{\Lambda}(X) \arrow[d, "p"] \\
\PP_Z \arrow[r, hook] &\mathrm{Hom}(\Lambda, \mathbb{C})
\end{tikzcd}
\]
where $\sigma \in U_{\sigma}$ and $p_{\sigma}$ is a covering map. 
\begin{itemize}
    \item 
$i(Z',0)$ has central charge $Z'$ and there is an equality of stability conditions
\[
i(Z',\phi') = e^{2\pi i (\phi' - \phi)}\cdot i(Z', \phi). 
\]
\item 
$p_{\sigma}(Z',\phi) = e^{2\pi i\phi}\cdot Z'$.
\end{itemize}
\item 
Every stability condition in $U_{\sigma}$ satisfies the support property with respect to $Q$.
\end{enumerate}
\end{theorem}
In the case of $X = E^{3}$, we may establish the following stronger statements characterizing the subset $\PP_Z$.
\begin{lemma}\label{lem:PZimage}
Let $\sigma \in \mathrm{Stab}_{\Lambda}(X)$ be a stability condition with central charge $Z$. Then the embedding of Theorem~\ref{thm:bayerdef} induces a homeomorphism:
\[
U_{\sigma} \simeq p^{-1}(\PP_Z).
\]
In addition, the subset $\PP_Z$ lies in the image of the embedding
    \[
\begin{tikzcd}
    \UU^+(3) \arrow[r, hook,"j"] & \mathrm{Hom}(\Lambda,\mathbb{C}).
\end{tikzcd}
\]
of Proposition~\ref{prop:HodgePrimitive}. 
\end{lemma}
\begin{proof}
For the first statement, Theorem~\ref{thm:bayerdef}(1) induces an injection
\[
U_{\sigma} \coloneqq (\PP_Z)_{\varphi = 0} \times \mathbb{R} \xhookrightarrow{} p^{-1}(\PP_Z)
\]
and it suffices to prove that this is surjective. If $\sigma' \in p^{-1}(\PP_Z)$, then the corresponding central charge $Z' \in \PP_Z$. Again by Theorem~\ref{thm:bayerdef}(1), there exists a stability condition $\sigma''$ with central charge $Z''$ such that
\[
Z'' = Z', \quad \varphi''( \OO_p) = \varphi'( \OO_p)
\]
By Theorem~\ref{thm:injective}, there is an equality of stability conditions $\sigma'' = \sigma'$, and the conclusion follows.

For the second statement, assume that $Z' \in \PP_Z$ is not contained in the image. By Theorem~\ref{thm:bayerdef}(1), there exists a stability condition $\sigma_{Z'} \in \mathrm{Stab}(X)$ with central charge $Z'$. On the other hand, Proposition~\ref{prop:central} immediately implies that this is a contradiction. 
\end{proof}

We fix a Lagrangian subset $L\subseteq \mb C^n$, and given a subspace $\UU \subseteq \UU^+(n)$ closed under the $U(1)$ action, we define the following:
\[
\UU_{\theta} \coloneqq \left\{ \Omega \in \UU \mid \mathrm{arg}(\Omega \vert_{L}) = \theta \right\}.
\]
\begin{lemma}\label{lem:covering}
The map
\[
\pi:\UU^+(n)_{\theta} \times \mathbb{R} \longrightarrow \UU^+(n), \qquad \pi( \Omega, \phi)= \mathrm{exp}(i \phi)\cdot \Omega
\]
is a universal covering of $\UU^+(n)$.

In addition, let $\UU \subseteq \UU^+(n)$ be a connected open subset closed under the $U(1)$ action. Then there is an equality:
\[
\pi^{-1}(\UU) = \UU_{\theta} \times \mathbb{R}.
\]
\end{lemma}
\begin{proof}
Note that $\Omega \vert_{L}\in{\bigwedge}^nL^*\otimes\mb C$ lives in the complexification of a 1-dimensional vector space over $\mb R$, so $\mathrm{arg}(\Omega \vert_{L})\in\mb R/2\pi\mb Z=S^1$ is well-defined.
Moreover the map 
\[
q:\mc U^+(n)\longrightarrow S^1,\qquad \Omega\mapsto \mathrm{arg}(\Omega \vert_{L})
\] 
is a trivial fiber bundle, since it is equivariant with respect to the $U(1)$ action on $\mc U^+(n)$ and $S^1$ given by scalar multiplication. Thus, $\mc U^+(n)=\UU^+(n)_{\theta} \times S^1$.
It follows from \cite[Proposition 2.5]{haiden20} that $q$ is a homotopy equivalence and thus its fibers, in particular $\UU^+(n)_{\theta}$, are contractible, which implies the first statement.

The second statement follows, since under the assumption on $\UU$, $\pi^{-1}(\UU)$ is closed under the $\mb R$ action by translation on the second factor.
\end{proof}

\subsection{Gluing of open embeddings}
\begin{lemma}\label{lem:liftfund}
Let $\sigma \in \mathrm{Stab}_{\Lambda}(X)$ be a stability condition with central charge $Z$. Then there exists a connected open neighborhood $\UU_{Z}\subseteq \UU^+(3)$ of $Z$, with a commutative diagram
\[
\begin{tikzcd}
 \pi^{-1}(\UU_{Z})\arrow[r, "\sim"] \arrow[d]&  p^{-1}(\PP_Z) \subseteq \mathrm{Stab}_\Lambda(X)\arrow[d, shift right = 12,"p"]\\
  \UU_{Z}\arrow[r, "\sim"]  & \PP_{Z} \subseteq \mathrm{Hom}(\Lambda,\mathbb{C}) 
\end{tikzcd}
\]
\end{lemma}
\begin{proof}
Consider the following diagram
\begin{equation}\label{eqn:square}
\begin{tikzcd}
  (\UU_{Z})_{\varphi} \times \mathbb{R} \simeq   \pi^{-1}(\UU_{Z}) \arrow[r, "\sim","i"', dashed] \arrow[d]&(\PP_{Z})_{\varphi} \times \mathbb{R} \simeq p^{-1}(\PP_Z) \arrow[d,"p"] 
    \arrow[r, hook] & \mathrm{Stab}_{\Lambda}(X)\arrow[d] \\
  \UU_{Z}\arrow[r,"\sim", "j"']  & \PP_{Z} \arrow[r, hook]& \mathrm{Hom}(\Lambda,\mathbb{C}) 
\end{tikzcd}
\end{equation}
where the existence and commutativity of the right square follows from Theorem~\ref{thm:bayerdef}. The identification 
\[
(\PP_{Z})_{\varphi} \times \mathbb{R} \simeq p^{-1}(\PP_Z) 
\]
follows directly from Lemma~\ref{lem:PZimage}; we define the subset $\UU_Z$ as the inverse image of the embedding $j$ of Lemma~\ref{lem:PZimage}, which restricts to a homeomorphism. The homeomorphism
\[
 (\UU_{Z})_{\varphi} \times \mathbb{R} \simeq  \pi^{-1}(\UU_{Z}) 
\]
follows from Lemma~\ref{lem:covering}, and it suffices to show the existence of a homeomorphism $i$ together with the commutativity of the left square.

We define the morphism $i$ by
\begin{align*}
 i \colon   (\UU_{Z})_{\varphi} \times \mathbb{R} &\rightarrow (\PP_{Z})_{\varphi} \times \mathbb{R}\\
    (\Omega, \phi) &\mapsto (j(\Omega), \phi)
\end{align*}
which implies the commutativity of the left square, from which the conclusion follows.
\end{proof}

\begin{lemma}\label{lem:glue}
Let $\{ \sigma_i \} \subset \mathrm{Stab}_{\Lambda}(X)$ be stability conditions with central charges $Z_i$. Then the corresponding commutative diagrams of Lemma~\ref{lem:liftfund} glue to give a commutative diagram
\[
\begin{tikzcd}
\pi^{-1}(\UU)\arrow[r, "\sim","i"'] \arrow[d,"\pi"]&  p^{-1}(\PP) \subseteq \mathrm{Stab}_\Lambda(X)\arrow[d, shift right = 12,"p"]\\
  \UU\arrow[r, "\sim", "j"']  & \PP \subseteq \mathrm{Hom}(\Lambda,\mathbb{C}) 
\end{tikzcd}
\]
where $\PP = \bigcup\limits_{i} \PP_{Z_i}$, and $\UU  = \bigcup\limits_{i} \UU_{Z_i}$.
\end{lemma}
\begin{proof}
Defining $\PP = \bigcup\limits_{i} \PP_{Z_i}$, and $\UU  = \bigcup\limits_{i} \UU_{Z_i}$, we have the commutative diagram
\[
\begin{tikzcd}
    \pi^{-1}(\UU) \arrow[r, "\sim","i"', dashed] \arrow[d,"\pi"]& p^{-1}(\PP) \arrow[d,"p"] 
    \arrow[r, hook] & \mathrm{Stab}_{\Lambda}(X)\arrow[d] \\
  \UU\arrow[r,"\sim", "j"']  & \PP \arrow[r, hook]& \mathrm{Hom}(\Lambda,\mathbb{C}) 
\end{tikzcd}
\]
where the right square clearly commutes by definition. It suffices to show that all four morphisms of Lemma~\ref{lem:liftfund} agree on their respective overlaps and hence that $\pi, p, j$, and $i$ are well-defined, and that the left square commutes. The claim is clear for the morphisms $\pi, p$, and $j$, as these are induced from the global morphisms 
\[
\widetilde{\UU^+(3)} \rightarrow \UU^+(3), \quad \mathrm{Stab}_{\Lambda}(X) \rightarrow \mathrm{Hom}(\Lambda,\mathbb{C}), \quad \UU^+(3) \xhookrightarrow{} \mathrm{Hom}(\Lambda,\mathbb{C})
\]
respectively. For the morphism $i$, it suffices to show that given two elements
\[
\Omega_1 \in \pi^{-1}(\UU_{Z_1}), \ \Omega_2 \in \pi^{-1}(\UU_{Z_2}), \quad \Omega_1 = \Omega_2,
\]
these induce the same stability conditions in $\mathrm{Stab}_{\Lambda}(X)$. By the construction in the proof of Lemma~\ref{lem:liftfund}, these induce stability conditions $\sigma_1, \sigma_2$ with the same central charges and phases for the skyscraper sheaf, and so $\sigma_1 = \sigma_2$ by Theorem~\ref{thm:injective}, and the corresponding morphisms glue.

Finally, the left diagram commutes as inverse images commute with restrictions, together with the commutativity of Lemma~\ref{lem:liftfund}.
\end{proof}

\subsection{Proof of Theorem~\ref{thm:main}}

Given a stability condition $\sigma = ( Z, \PP)\in \mathrm{Stab}_{\Lambda}(X)$ with quadratic form $Q$ and a transformation $g \in \mathrm{Sp}(6,\mathbb{Q})$, we constructed a new stability condition $g \cdot \sigma$ in Proposition~\ref{prop:inducedstab} as follows:
\[
Z' \coloneqq Z \circ g, \quad \PP'(\varphi) \coloneqq \{ \EE \in D^b(X)\  \vert \ g(\EE) \in \PP(\varphi)\}, \quad Q' \coloneqq Q \circ g
\]
\begin{lemma}\label{lem:gaction}
    Let $\sigma \in \mathrm{Stab}_{\Lambda}(X)$ be a stability condition with central charge $Z$, quadratic form $Q$, and $g \in \mathrm{Sp}(6,\mathbb{Q})$. Then there is an equality of sets
    \[
\PP_{g\cdot Z} = g \cdot \PP_Z.
    \]
In particular, this implies
\[
 \UU_{g \cdot Z} = g \cdot \UU_Z.
\]
\end{lemma}
\begin{proof}
By definition, $\PP_{g \cdot Z}$ can be identified with the subset of central charges $Z'$ such that if $v \in \mathrm{ker}(Z')$, then $Q(g(v)) < 0$. Let $Z' \in \PP_{g \cdot Z}$. We claim that $Z' \circ g^{-1} \in \PP_Z$. Indeed, if $v \in \mathrm{ker}(Z' \circ g^{-1})$, then $g^{-1}(v) \in \mathrm{ker}(Z')$ and consequently,
\[
(Q \circ g)(g^{-1} v) = Q(v) < 0
\]
implying the claim. Conversely, assume that $Z' \in g \cdot \PP_Z$, and we claim that $Z' \in \PP_{g \cdot Z}$. By definition, this implies that $g^{-1} \cdot Z' \in \PP_{Z}$. Thus, if $v$ satisfied $Z' \circ g^{-1}(v) = 0$, then $Q(v) < 0$. To verify the claim, let $w \in \mathrm{ker}(Z')$; this implies that $g(w) \in \mathrm{ker}(Z' \circ g^{-1})$ and we immediately obtain $Q(g(w)) < 0$, giving the equality of sets:
    \[
\PP_{g\cdot Z} = g \cdot \PP_Z.
    \]

By Lemma~\ref{lem:PZimage}, we have isomorphisms:
\begin{equation}\label{eqn:gziso}
\UU_Z \simeq \PP_Z, \quad \UU_{g\cdot Z} \simeq \PP_{g \cdot Z} = g \cdot \PP_Z.
\end{equation}
By Lemma~\ref{lem:Umirror}, this identification is $\mathrm{Sp}(6,\mb Q)$-equivariant, and so the first isomorphism implies
\[
g \cdot \UU_Z \simeq g \cdot \PP_Z.
\]
Combining this with the second equation in (\ref{eqn:gziso}) implies the conclusion.
\end{proof}

To conclude the proof of Theorem~\ref{thm:main}, we recall the mirror isomorphism of~(\ref{eq:totalHiso}) and its restriction via Proposition~\ref{prop:HodgePrimitive}.
\[
\beta \colon \NN(E^n) \otimes \mathbb{C} \xlongrightarrow{\sim} H^n_{\mathrm{pr}}((E^\circ)^n;\mb C)
\]
We note that the central charge of any stability condition must lie in the pre-image under $\beta$ of the subspace $\UU^+(n)$.
\begin{prop}\label{prop:central}
Assume that $Z \not\in \beta^{-1}(\mc U^+(n))$. Then $Z$ cannot be the central charge of a stability condition on $X$.
\end{prop}

\begin{proof}
Assume on the contrary that there exists a stability condition $\sigma$ with such a central charge $Z$, vanishing along the image under $\beta$ of the class $\Gamma$ of a real Lagrangian subspace. From the identification of the Lagrangian Grassmanian of \cite[Example 3.2.6 \& Section 3]{Polishchuk_2014}, together with \cite[Proposition 3.1.4]{Polishchuk_2014}, we conclude that the image under $\beta$ of the classes of rational Lagrangian subspaces must contain $\sigma$-semistable objects. In particular, $Z$ satisfies the support property with respect to all such classes. This implies that the image of $Z$ under $\beta$ gives an $n$-form $\Omega \in H^n_{pr}((E^o)^n; \mathbb{C})^\vee$ satisfying the support property with respect to rational Lagrangian subspaces, but vanishing along a real Lagrangian subspace, in contradiction with~\cite[Proposition 4.1]{haiden20}.
\end{proof}

Let $\Pi \subseteq \UU^+(3)$ be the fundamental domain.

\begin{proof}[Proof of Theorem~\ref{thm:main}]
We first claim that there exists stability conditions $\sigma$ with central charges $Z$, whose corresponding open neighborhoods $\UU_{Z}$ cover $\UU^+(3)$. By Proposition~\ref{prop:fullsupport}, Proposition~\ref{prop:quadratic}, and Theorem~\ref{thm:bayerdef}, there exists a stability condition corresponding to every element $Z \in \Pi$. In particular, the union over all open neighborhoods $\UU_Z$ gives an open subset $\UU \subseteq \UU^+(3)$ containing $\Pi$. By Lemma~\ref{lem:gaction} and Proposition~\ref{prop:fstability}, it suffices to show that $\mathrm{Sp}(6,\mathbb{Q}) \cdot \UU = \UU^+(3)$. On the other hand, this follows directly from Proposition~\ref{prop:densesub}, establishing the claim.

Together with Lemma~\ref{lem:glue}, we obtain a commutative diagram
\[
\begin{tikzcd}
\widetilde{\mathcal{U}^{+}(3)} \arrow[r, hook, "i"] \arrow[d, "\pi"] & \mathrm{Stab}_{\Lambda}(X) \arrow[d,"p"]\\
\mathcal{U}^+(3) \arrow[r, hook, "j"] & \mathrm{Hom}(\Lambda, \mathbb{C})
\end{tikzcd}
\]
where $i$ is an open embedding. We claim that this is an isomorphism onto a connected component, and we argue as in \cite[Theorem 9.1]{MR3573975}. Indeed, assume that there exists a stability condition $\sigma \in \partial\, \widetilde{\UU^+(3)} \subseteq \mathrm{Stab}_{\Lambda}(X)$. As $\pi$ is a covering map, the corresponding central charge $Z$ must lie on the boundary $\partial \, \UU^+(3) \subseteq \mathrm{Hom}(\Lambda,\mathbb{C})$. In particular, this implies that $Z \not\in \UU^+(3)$, which is a contradiction to Proposition~\ref{prop:central}, concluding the proof of the Theorem.
\end{proof}

\bibliographystyle{amsplain}
\bibliography{stability}
\end{document}